\documentclass[paper=a4, fontsize=11pt, abstract=on]{article}
\usepackage{amsmath}
\usepackage{amsthm}
\usepackage{amsfonts}
\usepackage{amssymb}
\usepackage{mathtools}
\usepackage{hyperref}
\usepackage{titling}
\usepackage{framed}
\usepackage{dsfont}
\usepackage{mathrsfs}
\usepackage{color}

\usepackage[headsepline]{scrlayer-scrpage}
\usepackage{geometry}

\usepackage[utf8]{inputenc}

\usepackage{lmodern}
\usepackage[T1]{fontenc}

\usepackage{etoolbox}
\patchcmd{\thebibliography}
{\settowidth}
{\setlength{\itemsep}{0pt plus 0.1pt}\settowidth}
{}{}
\apptocmd{\thebibliography}
{\small}
{}{}

\title{\textbf{ Optimal Variance--Gamma approximation\\ on the second Wiener chaos}}
\author{Ehsan Azmoodeh 
	\thanks{Department of Mathematical Sciences,
		University of Liverpool, Liverpool L69 7ZL, United Kingdom.
		E-mail: \texttt{ehsan.azmoodeh@liverpool.ac.uk}}, \quad 
	Peter Eichelsbacher
	\thanks{Ruhr University Bochum, Faculty of Mathematics, IB 2/115, 44780 Bochum, Germany. E-mail: \texttt{peter.eichelsbacher@rub.de}} \quad 
	Christoph  Th\"ale \thanks{Ruhr University Bochum, Faculty of Mathematics, IB 2/111, 44780 Bochum, Germany. E-mail:\texttt{ christoph.thaele@rub.de.}}
}

\date{\today}
\theoremstyle{plain}
\newtheorem{Thm}{Theorem}[section]
\newtheorem{thm}[Thm]{Theorem}

\newtheorem{lem}[Thm]{Lemma}
\newtheorem{Prop}[Thm]{Proposition}
\newtheorem{prop}[Thm]{Proposition}

\newtheorem{cor}[Thm]{Corollary}
\theoremstyle{definition}

\newtheorem{rem}[Thm]{Remark}

 \def\P{\mathbb{P}}
\def\E{\mathbb{E}}
\def\R{\mathbb{R}}

\def\N{\mathbb{N}}
\newcommand{\HH}{\mathfrak{H}}

\newcommand{\cont}[1]{\mathbin{\otimes_{#1}}}
\newcommand{\contIterated}[2]{\mathbin{\otimes_{#1}^{(#2)}}}

\newcommand{\CVG}{VG_c(r,\theta,\sigma)}

\DeclareMathOperator{\Var}{Var}
\DeclareMathOperator{\Ker}{Ker}
\DeclareMathOperator{\Tr}{Tr}

\DeclareMathOperator{\VG}{VG}

\DeclareMathOperator{\CenteredGamma}{\overline{\Gamma}}

\DeclarePairedDelimiter\sprod{\langle}{\rangle}
\DeclarePairedDelimiter\abs{\lvert}{\rvert}
\DeclarePairedDelimiter\norm{\lVert}{\rVert}

\let\temp\epsilon \let\epsilon\varepsilon \let\varepsilon\temp
\def\geq{\geqslant}
\def\leq{\leqslant}

\date{  }
\makeatletter
\let\@fnsymbol\@alph
\makeatother

\begin{document}
	\maketitle
	\begin{abstract}
In this paper, we consider a target random variable  $Y \sim \CVG$ distributed according to a centered Variance--Gamma distribution. For a generic random element $F=I_2(f)$ in the second Wiener chaos with $\E[F^2]= \E[Y^2]$ we establish a non-asymptotic optimal bound on the distance between $F$ and $Y$ in terms of the maximum of difference of the first six cumulants. This six moment theorem extends the celebrated optimal fourth moment theorem of I.\ Nourdin \& G.\ Peccati for normal approximation. The main body of our analysis constitutes a splitting technique for test functions in the Banach space of Lipschitz functions relying on the compactness of the Stein operator. The recent developments around Stein method for Variance--Gamma approximation by R.\ Gaunt play a significant role in our study. As an application we consider the generalized Rosenblatt process at the extreme critical exponent, first studied by S.\ Bai \& M.\ Taqqu.
	\end{abstract}

	\vskip0.3cm
	\noindent \textbf{Keywords}: Cumulant, generalized Rosenblatt process, Malliavin calculus, six moment theorem, Stein's method, Variance--Gamma approximation, Wasserstein distance, Wiener chaos \\
	\noindent 
	\textbf{MSC 2020}: 62E17, 60F05, 60G50, 60H07
	

\section{Introduction and results}

\subsection{Motivation and a summary of related results}

The Variance--Gamma (VG) probability distribution on $\R$ with parameters $r >0, \theta \in \R, \sigma >0$ and $\mu \in \R$ has probability density function given by 
\begin{equation}\label{eq:VGdensity}
p_{\text{VG}} (x; r, \theta,\sigma,\mu) = \frac{1}{\sigma \sqrt{\pi} \Gamma(\frac{r}{2})} e^{ \frac{\theta}{\sigma^2} (x-\mu)} \left( \frac{\vert x -\mu \vert}{2 \sqrt{\theta^2 + \sigma^2}} \right)^{\frac{r-1}{2}} K_{\frac{r-1}{2}} \left( \frac{\sqrt{\theta^2 + \sigma^2}}{\sigma^2} \vert x -\mu \vert \right),\qquad x\in\R,
\end{equation} 
where $K_\nu(x)$ is the modified Bessel function of the second kind, see \cite[Appendix A]{gaunt-VG-perfect-bounds}. We write  $Y \sim \VG(r,\theta,\sigma,\mu)$ when random variable $Y$ is distributed according to the density $p_{\text{VG}}$; for a detailed account and different parametrizations the reader is referred to \cite{Robert-Thesis}. When the parameter $\mu$ satisfies $\mu=-r \theta$, we write $\CVG$ to denote the centered VG distribution, i.e., the one with mean zero. The family of VG distributions contains several classical probability distributions as special or limiting cases. Examples include the normal (Gaussian), the gamma or the normal product distribution as well as the difference of gamma distributions, see \cite[Proposition 1.2]{Robert-Thesis} and \cite[p.\ 11]{VarianceGammaPaper}. The VG distribution has widely been applied in financial modelling \cite{Madan98, Madan90}, and for other applications, see \cite{Laplace-bible}.

The Malliavin-Stein method \cite{StMethOnWienChaos, n-pe-1} is a powerful technique to derive quantitative limit theorems on the Wiener (or the Poisson and the discrete Rademacher) space. Undoubtedly, the most striking result in this direction is the {fourth moment} theorem \cite{FmtOriginalReference} due to Nualart \& Peccati, which can be considered as the spring of the theory; we refer the reader to \cite{n-pe-1} as well as to Section \ref{sec:Background-Material} below for any unexplained notion evoked in the present section. For fixed $ p \ge 2$ it states that a sequence $(F_n : n \ge 1) \subseteq \mathscr{H}_p$ of elements in the fixed Wiener chaos of order $p$ converges in distribution towards a standard normal distribution if and only if, as $n\to \infty$,  $\E[F^2_n] \to 1$ and $\E[F^4_n] \to 3$. In \cite{StMethOnWienChaos}, Nourdin \& Peccati combined Stein method for normal approximation with Malliavin calculus to give a quantitative version of this fourth moment theorem. Their result reads as follows. Let $N \sim N(0,1)$ be a standard Gaussian random variable, and let $(F_n : n \ge 1) \subseteq \mathscr{H}_p$  be such that $\E[F_n^2]=1$ for each $n\geq 1$. Then, 
 \begin{equation}\label{eq:FMT-MS}
 d_{\text{TV}} (F_n, N)  \le  \sqrt{\frac{4p-4}{3p}} \, \sqrt{ \E[F_n^4] -3},
 \end{equation}
 where we recall that $d_{\text{TV}} (F,G):=  \sup_{B \in \mathcal{B}(\R)} \big \vert  \P (F \in B)  - \P (G \in B ) \big \vert$ denotes the total variation distance between two random elements $F$ and $G$. It was believed for years that the appearance of the square root in the bound \eqref{eq:FMT-MS} makes it sub-optimal. In fact, it is and we refer to \cite{OptBerryEsseenRates} for an optimal bound for normal approximation in a 'smooth' integral probability metric. Finally, Nourdin \& Peccati  \cite{n-p-optimal} established the optimal bound for the total variation distance. Their technique is based on an elegant iteration of Stein's method and the integration-by-parts formula (see \eqref{eq:IntegrationByParts}) from Malliavin calculus, resulting in so-called higher-order iterated Gamma operators (defined at \eqref{eq:GammOperatorDefinition}). Their result says that there are two constants $C_1,C_2>0$, independent of $n$, so that for any sequence $(F_n : n \ge 1) \subseteq \mathscr{H}_p$ of random elements in a fixed Wiener chaos of order $p\geq 2$,
 \begin{equation}\label{eq:Optimal-Normal}
 C_1\, \max \Big\{   \big \vert  \kappa_3 (F_n) \big \vert ,   \big \vert  \kappa_4 (F_n) -3 \big\vert         \Big \}  \, \le  \,  d_{\text{TV}}(F_n,N) \, \le C_2 \, \max \Big\{   \big \vert  \kappa_3 (F_n) \big \vert ,   \big \vert  \kappa_4 (F_n) -3 \big\vert         \Big \},
 \end{equation}
where $\kappa_3(F_n)$ and $\kappa_4(F_n)$ are the third and the fourth cumulant of $F_n$, respectively. In particular, when $p$ is odd and therefore $\kappa_3 (F_n)=0$, the optimal rate \eqref{eq:Optimal-Normal}  improves by a power of two the rate in \eqref{eq:FMT-MS}.

The Malliavin-Stein approach to probabilistic limit theorem on a fixed Wiener chaos has been extended to many other target distributions, for example to the (centered) {Gamma} distribution  \cite{n-p-noncentral, StMethOnWienChaos, InvPrinForHomSums,n-r, d-p}. For fixed $\nu >0$ let $G(\nu)$ stand for a random variable distributed according to the centered Gamma distribution on $\R$ with probability density function $p_\nu (x) = 2^{-\frac{\nu}{2}} \,\Gamma \left(\frac{\nu}{2} \right)^{-1} \, (x+ \nu)^{\frac{\nu}{2} -1} \, e^{- \frac{x+ \nu}{2}}$, $x\geq -\nu$. In \cite{Optimal-Gamma}, the authors provide an optimal rate analogous to \eqref{eq:Optimal-Normal} within the second Wiener chaos in terms of the maximum of the third and fourth cumulant. Namely, there are constants $C_1,C_2>0$ such that for a sequence $(F_n:n\geq 1)$ of elements $F_n \in \mathscr{H}_2$, 
\begin{multline}\label{eq:Gamma-Optimal}
C_1\, \max \Big\{  \big \vert \kappa_3 (F_n) - \kappa_3 (G(\nu)) \big \vert,  \big \vert \kappa_4 (F_n) - \kappa_4 (G(\nu)) \big \vert     \Big \} \\ \le  \,  d_{\mathcal{H}_2} (F_n,G(\nu)) \, \le \, C_2 \, \max \Big\{  \big \vert \kappa_3 (F_n) - \kappa_3 (G(\nu)) \big \vert,  \big \vert \kappa_4 (F_n) - \kappa_4 (G(\nu)) \big \vert     \Big \},
\end{multline}
where, following \cite{d-p, a-m-p-s}, for $k \ge 1$,  the so called {smooth Wasserstein} distance $d_{\mathcal{H}_k}(F,G)$ between two random elements $F$ and $G$ is given by
\begin{equation*}
d_{\mathcal{H}_k}(F,G) := \sup_{h \in \mathcal{H}_k} \abs[\Big]{ \E[h(F)] - \E[h(G)]  }. \label{eq:d_kMetricsDefinition}
\end{equation*}
Here, the class $\mathcal{H}_k$ of the test functions is defined as
$$
\mathcal{H}_k := \{ h \in C^{k-1}(\R) : h^{(k-1)} \in \operatorname{Lip}(\R) \text{ and } \norm{h^{(1)}}_\infty \leq 1,  \ldots, \norm{h^{(k)}}_\infty \leq 1 \},
$$
where $C^{k-1}(\R)$ is the space of $k-1$ times continuously differentiable functions on $\R$, $\operatorname{Lip}(\R)$ is the space of Lipschitz functions on $\R$ and $\norm{h^{(k)}}_{\infty}$ denotes the smallest Lipschitz constant of $h^{(k-1)}$, the derivative of order $k-1$ of $h$ (where for convenience $h^{(0)}=h$). It is worth pointing out that $ \norm{f^{(k)}}_\infty$ coincides with the uniform norm of the derivative of $f^{(k-1)}$, whenever $f^{(k-1)}$ is differentiable. In particular, for $k=1$ we obtain the classical {Wasserstein} distance. 

In \cite{VarianceGammaPaper} the Malliavin-Stein method for VG approximation has been introduced, we outline the general estimates in Section \ref{sec:MalliavinStein-VG}. Thanks to recent developments on Stein's method for the VG distribution by Gaunt \cite{gaunt-VG-perfect-bounds}, we present the results in a strengthened form using the distance $d_{\mathcal{H}_1}$ (instead of the so-called bounded Wasserstein distance $d_{\mathcal{H}_{b,1}}$ for which the results in \cite{VarianceGammaPaper} actually hold; note however that erroneously all bounds in \cite{VarianceGammaPaper} are formulated for the $d_{\mathcal{H}_1}$-distance). In contrast to the case of the normal or the Gamma distribution, the Stein equation for the VG distribution is a second-order differential equation. Unlike for the normal or the Gamma distribution, in the context of the Malliavin-Stein method this results in the appearance of the higher-order iterated Gamma operator $\Gamma_{alt,2}$, see Definition \ref{eq:GammaOperatorAltDefinition} below. Due to the involved nature of the $\Gamma_{alt,2}$ operator, up to this date only the case of the second Wiener chaos has been treated successfully, since in this case the general Malliavin-Stein bound could be translated into the  language of finitely many cumulants, see for example \cite[Conjecture 6.8]{MS-survey}. Let $Y \sim \CVG$ and assume that for each $n\geq 1$, $F_n \in \mathscr{H}_2$ is an element in the second Wiener chaos.  Then, the bound for VG approximation of $F_n$ in terms of cumulants reads as follows, see Theorem \ref{thm:MS-CVG-General-Upper-Bound}:
\begin{equation}\label{eq:VG-cumulants-bound-introduction}
d_{\mathcal{H}_1}(F_n,Y) \, \le \,  C \, \Bigg\{    \sqrt{ \sum_{\ell=2}^{6}  C_\ell \left( \kappa_\ell (F_n) - \kappa_\ell (Y)  \right) } +  \Big \vert   \kappa_3 (F_n) - \kappa_3 (Y) \Big \vert      \Bigg \},
\end{equation}
where $(C_\ell : \,  \ell =2,...,6)$ are explicit constants depending only to parameters $r, \theta$, and $\sigma$. The bound \eqref{eq:VG-cumulants-bound-introduction} and its comparison with \eqref{eq:FMT-MS} and \eqref{eq:Optimal-Normal} (for normal approximation) or \eqref{eq:Gamma-Optimal} (for Gamma approximation) are the main motivation and the starting point for the present paper.

\subsection{Statement of the main result}\label{sec:main-results}

Although bound \eqref{eq:VG-cumulants-bound-introduction} is very handy and shows that convergence of the first six cumulants already implies convergence in distribution towards a VG distribution, its similarity with the bound \eqref{eq:FMT-MS} for normal approximation raises the question whether also in this situation the appearance of the square root makes it sub-optimal. In fact, our main result shows that this is indeed the case; for the proof see Section \ref{sec:proof-Upper-Bound}.

\begin{thm}[\textbf{Optimal Variance--Gamma approximation}]\label{thm:Main-thm}
Let $Y \sim VG_c(r,\theta,\sigma)$ be distributed according to a centered Variance--Gamma distribution with parameters $r>0$, $\theta\in\R$ and $\sigma>0$. Let $(F_n:n\geq 1) \subseteq \mathscr{H}_2$ be a sequence of elements from the second Wiener chaos. 
Define 
\begin{equation}\label{eq:M(F)}
\mathbf{M}(F_n):= \max \Big\{  \Big \vert  \kappa_\ell (F_n) - \kappa_\ell(Y) \Big \vert \, : \, \ell=2,3,4,5,6 \Big \}.
\end{equation}
Then, there are constants $C_1,C_2>0$ only depending on $ r$, $\theta$ and $\sigma$ such that
\begin{equation}\label{eq:VG-Optimal-Rate}
C_1\, \mathbf{M}(F_n)  \,   \le   d_{\mathcal{H}_{2}} (F_n,Y) \, \le C_2 \, \mathbf{M} (F_n).
\end{equation}
\end{thm}

We recall an important stability result established by Nourdin \& Poly \cite{n-poly-2wiener} for convergence in distribution of elements in the second Wiener chaos.  It states that for a sequence $(F_n : n \ge 1)  \subseteq \mathscr{H}_2$ in the second Wiener chaos converging in distribution to a random element $F$, as $n\to\infty$, one necessarily has that $F \in \mathscr{H}_1 \oplus \mathscr{H}_2$, the direct sum of the first and the second Wiener chaoses. In combination with Theorem \ref{thm:Main-thm} this yields the following result.

\begin{cor}\label{cor:2chaos-targets}
	Fix $\alpha, \beta >0$ and  let $r \in \N$. Assume that $Y \in \mathscr{H}_2$ belongs to the second Wiener chaos and takes the form
	$$
	Y \stackrel{\text{law}}{=} \sum_{i=1}^{r} \alpha (N^2_i-1) - \sum_{j=1}^{r}\beta (\widetilde{N}^2_j-1)  \sim  \VG_c (r, \alpha - \beta, 2 \sqrt{\alpha\beta}),
	$$
	where $(N_i, \widetilde{N}_i : \, i \ge 1)$ are independent standard Gaussian random variables. Let $(F_n:n\geq 1) \subseteq \mathscr{H}_2$ be a sequence of elements belonging to the second Wiener chaos. Then, there are constants $C_1,C_2>0$ depending only on $\alpha$, $\beta$ and $r$ such that
	\begin{equation*}
		C_1\, \mathbf{M}(F_n)   \, \le  \,  d_{\mathcal{H}_{2}} (F_n,Y) \,  \le C_2 \, \mathbf{M} (F_n).
	\end{equation*}
\end{cor}

We close this section with a number of comments related to Theorem \ref{thm:Main-thm}.

\begin{rem}\label{rem:main-thm}
\begin{itemize}	\setlength\itemsep{-0.3em}
	\item[(a)]  In most practical applications, the rate \eqref{eq:VG-Optimal-Rate} provides an improvement of \eqref{eq:VG-cumulants-bound-introduction}, for a concrete example see Section\ref{sec:Bai-Taqqu}. In addition, when $Y \sim \text{SVG}_c (r, \theta =0, \sigma)$ has a centered and \textit{symmetric} VG distribution, the quantity $\mathbf{M}(F_n)$ in \eqref{eq:VG-Optimal-Rate} can be replaced by
	$$
	\mathbf{M}^{\prime} (F_n)  := \max \big\{  \big \vert  \kappa_\ell (F_n) - \kappa_\ell(Y) \big \vert \, : \, \ell=2,3,4,6 \big \}
	$$ 
	where in comparison to $\mathbf{M}(F_n)$ the fifth cumulant difference $\kappa_5(F_n)-\kappa_5(Y)$ does not appear.
	However, the presence of the difference $\kappa_3(F_n)-\kappa_3(Y)$ is inevitable as demonstrated by \cite[Example 2.1]{a-a-p-s}. 
	
	\item[(b)] A significant feature of the optimal rate \eqref{eq:VG-Optimal-Rate} is that it is non asymptotic and a priori does not assume the law of the chaotic random variables $F_n\in \mathscr{H}_2$ in the second Wiener chaos to be close to that of the Variance--Gamma target distribution $Y$. The same phenomenon has first been observed for the {centered Gamma distribution} in \cite{Optimal-Gamma}. 

	\item[(c)] For the proof of the upper bound in Theorem \ref{thm:Main-thm}, the starting point is an adaption of the technique developed in \cite{n-p-optimal} for normal approximation. However, in order to achieve the optimal upper bound we re-employ a novel operator theory technique introduced in \cite{Optimal-Gamma}  within Stein's method to split appropriate test functions. This is the topic of Section \ref{sec:operator-theory}. Our methodology to obtain the optimal lower bound is based on complex analysis and differs from that in \cite{n-p-optimal}. 

	\item[(d)] Due to the existence of singularities in the derivative of the solution of the Variance--Gamma Stein equation \cite[Propositions 3.4, 3.5]{gaunt-VG-perfect-bounds} associated to non smooth test functions (such as indicator functions), an optimal rate for non-smooth probability metrics such as the {Kolmogorov} or the total variation distance is out of the scope of the techniques developed in the present paper.  
	
	\item[(e)] Let us briefly comment on a natural thought relating to the generalization of the optimal rate \eqref{eq:VG-Optimal-Rate} to higher order Wiener chaoses. In fact, at least for the upper bound which is more demanding part, such an extension would come at the cost of understanding two technically rather involved computations. The first would be a tractable relation between the iterated Gamma operators $\Gamma_{alt,2}$ and $\Gamma_2$. For elements of the second Wiener chaos we overcome this difficulty in Proposition \ref{Prop:RelationOldAndNewGamma}. Second, in higher order chaoses, verifying the crucial variance estimates \eqref{eq:Variance-Estimate-1}, and \eqref{eq:ESS-VE2} turns into estimates involving norms of {contraction operators}, which were not feasible for us. Furthermore, our method to achieve the optimal lower bound, relying on complex analysis, cannot be used in higher order chaoses, and hence require the development of new ideas. 
	
	\item[(f)] In \cite[Corollary 4.2]{gaunt-normal-product}, Gaunt derived the optimal $1/n$ rate for normal product approximation, which is a particular distribution in the VG class. Although his result allows for a wide class of potential applications (not necessary restricted to Gaussian set-up, i.e., the Wiener chaos), however imposing the normal distribution to the entries of the partial sums considered in \cite{gaunt-normal-product} leads to a special subclass of the second Wiener chaos. We would like to emphasize at this point that this framework is still not broad enough to capture the application we develop in Section \ref{sec:Bai-Taqqu}. Furthermore, in \cite{gaunt-normal-product} stronger smoothness assumptions on the test functions defining the probability metric are required.
\end{itemize}
\end{rem}

\subsection{Application to the generalized Rosenblatt process at extreme critical exponents} \label{sec:Bai-Taqqu}

Recently Bai \&  Taqqu  \cite{Bai-Taqqu} considered the so called {generalized Rosenblatt} process $F_{\gamma_1,\gamma_2}  =   ( F_{\gamma_1,\gamma_2}(t)  :  t \ge 0 )$ which is defined a a double Wiener-It\^o stochastic integral 
\begin{equation}\label{eq:GRP}
F_{\gamma_1,\gamma_2} (t)  = A \int_{\R^2}^{\prime}  \left(    \int_{0}^{t} (s-x_1)^{\gamma_1} (s-x_2 )^{\gamma_2} ds \right )  B(dx_1)B(dx_2),
\end{equation}
where the prime $\prime$ indicates exclusion of the diagonal $\{x_1=x_2\}$ in the stochastic integral, $B$ denotes a standard Brownian random measure on $\R^2$, and $A\neq0$ is a normalizing constant so that $\E[F_{\gamma_1,\gamma_2}(1) ^2]=1$. The exponents $\gamma_1$, and $\gamma_2$ belong to the open triangle $\Delta$ given by
\begin{equation*}
\Delta : =  \Big\{    (\gamma_1,\gamma_2) \in \R^2  \, :  \,  \gamma_i \in (-1,-1/2), i=1,2, \,  \gamma_1 + \gamma_2  > -3/2           \Big \}.
\end{equation*}
The definition of $\Delta$ ensures that the integrand in \eqref{eq:GRP} belongs to $L^2(\R^2)$, and hence the process is well-defined. The special case $F_{\gamma} (t) : =  F_{\gamma,\gamma}(t)$, where $\gamma \in (-3/4,-1.2)$ is the well-known Rosenblatt process \cite{Taqqu1975}. The self-similarity property of $F_{\gamma_1,\gamma_2}$ yields that  $F_{\gamma_1,\gamma_2} (t)  \overset{\text{law}}{=} t^{\gamma_1 + \gamma_2 +2}F_{\gamma_1,\gamma_2} (1)$. Hence, hereafter we work only with random variable $F_{\gamma_1,\gamma_2}: =  F_{\gamma_1,\gamma_2}(1)$. 

In \cite{Bai-Taqqu}, Bai \& Taqqu studied  distributional behavior of the random variable $F_{\gamma_1,\gamma_2}$ at extreme critical exponents, that is, when the exponents $\gamma_1$, and $\gamma_2$ approach the boundaries of the triangle  $\Delta$, and obtained (among other results) the following remarkable limit theorem (the last assertion in part (b) follows from the results in \cite[Section 3.3]{a-a-p-s}). 

\begin{thm}[\cite{Bai-Taqqu}, Theorem 2.2 and Theorem 2.4]  \label{thm:BT}
Consider the sequence of random variables $(F_{\gamma_1,\gamma_2}  :  (\gamma_1,\gamma_2) \in \Delta) $.
\begin{itemize}
	\item[(a)]  As $(\gamma_1,\gamma_2) \to (-1/2,\gamma)$ or $(\gamma_1,\gamma_2) \to (\gamma,-1/2)$, where $-1 < \gamma < -1/2$, the following weak convergence takes place:
	\begin{equation*}
	F_{\gamma_1,\gamma_2}    \stackrel{\text{law}}{\longrightarrow} Y \stackrel{\text{law}}{=} N_1 \times N_2  \sim  \VG_c (1,0,1)
	\end{equation*}
	where $N_1, N_2  \sim N(0,1)$ are independent standard Gaussian random variables. Moreover, there exists a constant $C>0$ so that, as $\gamma_1 \to -1/2$, 
	\begin{equation}\label{eq:BT-rate-item(a)}
	d_{\mathcal{H}_1} (F_{\gamma_1,\gamma_2} , Y)  \le C  \sqrt{-\gamma_1 - \frac{1}{2}}.
	\end{equation}
	\item[(b)] Let $\rho \in (0,1)$, and consider random variable 
	\begin{equation}\label{eq:BT-target-interesting-corner}
	Y_\rho    =  \frac{\alpha_\rho}{\sqrt{2}}  (N^2_1  -1)  - \frac{\beta_\rho}{\sqrt{2}} (N^2_2 -1)  \sim   \VG_c(1,\frac{\alpha_\rho - \beta_\rho}{\sqrt{2}},  \sqrt{2\alpha_\rho  \beta_\rho}),
	\end{equation} 
	where $N_1, N_2 \sim N(0,1)$ are independent standard Gaussian random variables. Put
	\begin{equation*}
	\alpha_\rho : =  \frac{ (2 \sqrt{\rho})^{-1}+  (\rho +1)^{-1}   }{  \sqrt{   (2\rho)^{-1}  + 2 (\rho+1)^{-2}   }}, \quad   \beta_\rho := \frac{ (2 \sqrt{\rho})^{-1}-  (\rho +1)^{-1}   }{  \sqrt{   (2\rho)^{-1}  + 2 (\rho+1)^{-2}  }}.
	 \end{equation*}
Assume $\gamma_1 \ge \gamma_2$, and that $\gamma_2  =  \frac{\gamma_1 + 1/2}{\rho} - 1/2$.  Then, as $\gamma_1  \to -1/2$ (and hence, $\gamma_2 \to -1/2$ too), the following weak convergence takes place:
\begin{equation*}
F_{\gamma_1,\gamma_2}    \stackrel{\text{law}}{\longrightarrow}  Y_\rho.
\end{equation*}
Moreover, the exists a  constant $C>0$ so that,  as $\gamma_1 \to -1/2$,
\begin{equation} \label{eq:BT-rate-item(b)}
d_{\mathcal{H}_1}  (F_{\gamma_1,\gamma_2},  Y_\rho )    \le C  \sqrt{-\gamma_1 - \frac{1}{2}}.
\end{equation}		
\end{itemize}
\end{thm}

Since by definition all the random variables $F_{\gamma_1,\gamma_2}$ are elements of the second Wiener chaos $\mathscr{H}_2$, we can apply our main Theorem \ref{thm:Main-thm} in order to deduce the following improved and in fact optimal rate of convergence in the Bai-Taqqu limit theorem for the $d_{\mathcal{H}_2}$-distance; see Section \ref{sec:proof:BaiTaqqu} for the proof.

\begin{thm}\label{thm:BT-OptimalRate}
	Let all the assumption of Theorem \ref{thm:BT} prevail.  In both cases of Theorem \ref{thm:BT} and with the corresponding target random variable $Y$, as $\gamma_1 \to -1/2$, there exist two constants $C_1,C_2>0$ such that
	\begin{equation*}
	C_1 \,    \left \vert -\gamma_1 - \frac{1}{2}  \right \vert   \,  \le \,  d_{\mathcal{H}_2}\left( F_{\gamma_1,\gamma_2}, Y  \right)  \, \le \, C_2  \, \left \vert -\gamma_1 - \frac{1}{2}  \right \vert .
	\end{equation*}
\end{thm}

\begin{rem} 
\begin{itemize}	\setlength\itemsep{-0.3em}
\item[(a)] The rate $\sqrt{-\gamma_1 - 1/2}$ at item (a) in Theorem \ref{thm:BT} is first obtained  in \cite{Bai-Taqqu}, and recently in \cite{gaunt-VG-perfect-bounds} in the stronger probability metric $d_{\mathcal{H}_1}$.  On the other hand, the same rate $\sqrt{-\gamma_1 - 1/2}$ appearing at item (b) was obtained first in \cite{a-a-p-s} for the $2$-Wasserstein distance (see \cite[Definition 1.1]{a-a-p-s}) by using a purely discrete Hilbert space approach, whereas Gaunt recently in \cite{gaunt-VG-perfect-bounds} established the same rate for the distance $d_{\mathcal{H}_1}$ using Stein's method for VG approximation. Gaunt's approach has the following advantages. First, his approach is general, meaning that it provides estimate for a general random variable $F$ satisfying only a minimal set of assumptions, while the approach in \cite{a-a-p-s} can be applied only to elements living in the second Wiener chaos.  Second, it provides accesses to explicit constants, see  Theorem  \ref{thm:MS-CVG-General-Upper-Bound}.
\item[(b)] Gaunt in \cite[Proposition 3.6]{gaunt-VG-perfect-bounds} relates the distance $d_{\mathcal{H}_1}$ to the classical {Kolmogorov} distance $d_{\text{Kol}} (F,G):=  \sup_{x \in \R}  \big \vert   \P (F \le x)  - \P(G\le x) \big \vert  $ between two random variables $F$ and $G$. Applying his result, in the setting of Theorem \ref{thm:BT}, we infer that, as $\gamma_1 \to -1/2$,  
\begin{equation}\label{eq:bai-Taqqu-Kol}
d_{\text{Kol}} (F_{\gamma_1,\gamma_2} ,  Y)  \, \le  \,  C \left \vert    -\gamma_1  - \frac{1}{2}\right \vert ^{1/4}  \log  \left \vert      \frac{1}{ -\gamma_1  - \frac{1}{2}}   \right \vert,
\end{equation}
for some constant $C>0$ depending on the parameters of the target distribution, see  \cite[p.\ 18]{gaunt-VG-perfect-bounds} and \cite[Theorem 3]{a-m-p-s}. We strongly believe that this bound is subpotimal, but as explained in Remark \ref{rem:main-thm} (d) an improvement by our techniques seems currently out of reach.
\end{itemize}
\end{rem}

\section{Background material}\label{sec:Background-Material}

\subsection{Elements of Malliavin calculus on Wiener space}\label{sec:Malliavin-Operators}

In this section, we provide a brief introduction to Malliavin calculus and define some of the operators used in this framework. For more details, we refer the reader to the textbooks \cite{n-pe-1,GelbesBuch,Nua-Nua}.

\subsubsection{Isonormal Gaussian processes and Wiener chaos}
Let $\mathfrak{H}$ be a real separable Hilbert space with inner product $\sprod{\cdot,\cdot}_{\mathfrak{H}}$, and $X = \{X(h) : h \in \mathfrak{H} \}$ be an isonormal Gaussian process, defined on some probability space $(\Omega, \mathscr{F},P)$. This means that $X$ is a family of centered, jointly Gaussian random variables satisfying $\E[X(g)X(h)] = \sprod{g,h}_{\mathfrak{H}}$. We assume that $\mathscr{F}$ is the $\sigma$-algebra generated by $X$. For an integer $q \geq 1$, we write $\mathfrak{H}^{\otimes q}$ or $\mathfrak{H}^{\odot q}$ to denote the $q$-th tensor product or the  $q$-th symmetric tensor product of $\mathfrak{H}$, respectively. If $H_q(x) =(-1)^{q}e^{x^{2}/2}{\frac {d^q}{d x^n}}e^{-x^{2}/2}$ is the $q$-th Hermite polynomial, then the closed linear subspace of $L^2(\Omega)$ generated by the family $\{H_q(X(h)) : h \in \mathfrak{H}, \norm{h}_\mathfrak{H} = 1 \}$ is called the $q$-th \textit{Wiener chaos} of $X$ and will be denoted by $\mathscr{H}_q$. For $f \in \mathfrak{H}^{\odot q}$, let $I_q(f)$ be the $q$-th multiple Wiener-Itô integral of $f$. An important observation is that for any $f \in \mathfrak{H}$ with $\norm{f}_{\mathfrak{H}}=1$ we have that $H_q(X(f)) = I_q(f^{\otimes q})$. As a consequence, $I_q$ provides an isometry between $\mathfrak{H}^{\odot q}$ and the $q$-th Wiener chaos $\mathscr{H}_q$ of $X$. It is a well-known fact, called the \textit{Wiener-Itô chaotic decomposition}, that any element $F \in L^2(\Omega)$ admits the expansion
\begin{equation}
F = \sum_{q=0}^{\infty} I_q(f_q), \label{eq:ChaoticExpansion}
\end{equation}
where $f_0 = \E[F]$ and the $f_q \in \mathfrak{H}^{\odot q}$, $q \geq 1$ are uniquely determined.

Let $(e_{k},\,k\geq 1)$ be a complete orthonormal system in $\HH$. Given $f\in  \HH^{\odot p}$ and $g\in \HH^{\odot q}$, for every
$r=0,\ldots ,p\wedge q$ where $p\wedge q$ denotes the minimum of $p$ and $q$, the \textit{contraction} of $f$ and $g$ of order $r$
is the element of $ \HH^{\otimes (p+q-2r)}$ given by
\begin{equation}
f\otimes _{r}g=\sum_{i_{1},\ldots ,i_{r}=1}^{\infty }\langle
f,e_{i_{1}}\otimes \ldots \otimes e_{i_{r}}\rangle _{ \HH^{\otimes
		r}}\otimes \langle g,e_{i_{1}}\otimes \ldots \otimes e_{i_{r}}
\rangle_{ \HH^{\otimes r}}.  \label{v2}
\end{equation}
Notice that the definition of $f\otimes_r g$ does not depend
on the particular choice of $(e_k,\,k\geq 1)$. Also note that
$f\otimes _{r}g$ is not necessarily symmetric, its
symmetrization will be denoted by $f\widetilde{\otimes }_{r}g\in  \HH^{\odot (p+q-2r)}$.
Moreover, $f\otimes _{0}g=f\otimes g$ equals the tensor product of $f$ and
$g$ while, for $p=q$, $f\otimes _{q}g=\langle f,g\rangle _{ \HH^{\otimes q}}$. 
Contractions appear naturally in the \emph{product formula} for multiple Wiene-It\^o integrals. Namely, if $f\in  \HH^{\odot p}$ and $g\in  \HH^{\odot q}$, then
\begin{eqnarray}\label{multiplication}
I_p(f) I_q(g) = \sum_{r=0}^{p \wedge q} r! {p \choose r}{ q \choose r} I_{p+q-2r} (f\widetilde{\otimes}_{r}g).
\end{eqnarray}
Another important result is the following \textit{isometry} property of multiple integrals. Let $f \in \mathfrak{H}^{\odot p}$ and $g \in \mathfrak{H}^{\odot q}$, where $1 \leq q \leq p$. Then
\begin{equation}
\E[ I_p(f) I_q(g) ] = \begin{cases}
p! \, \sprod{f,g}_{\mathfrak{H}^{\otimes p}}  & \text{if } p= q\\
0 & \text{otherwise}.
\end{cases} \label{eq:IsometryProperty}
\end{equation}

\subsubsection{Malliavin operators}

We denote by $\mathscr{S}$ the set of \textit{smooth} random variables, that is, random variables of the form $F= g(X(\varphi_1),\ldots, X(\varphi_n))$, where $n \geq 1$, $\varphi_1, \ldots, \varphi_n \in \HH$ and $g:\R^n \to \R$ is a $C^{\infty}$-function, whose partial derivatives have at most polynomial growth. For such random variables, we define the \textit{Malliavin derivative} of $F$ with respect to $X$ as the $\HH$-valued random element $DF \in L^2(\Omega,\HH)$ given by
\[ DF = \sum_{i=1}^{\infty} \frac{\partial g}{\partial x_i} \big( X(\varphi_1), \ldots, X(\varphi_n) \big) \, \varphi_i. \]
The set $\mathscr{S}$ is dense in $L^2(\Omega)$ and using a closure argument, one can extend the domain of $D$ to $\mathbb{D}^{1,2}$, the closure of $\mathscr{S}$ in $L^2(\Omega)$ with respect to the norm $\norm{F}_{\mathbb{D}^{1,2}} := \E[F^2] + \E[ \norm{DF}_{\HH}^2 ]$. We refer to \cite{n-pe-1} for a more general definition of higher order Malliavin derivatives and the spaces $\mathbb{D}^{p,q}$. The Malliavin derivative satisfies the following \textit{chain-rule}. If $\phi:\R^m \to \R$ is a continuously differentiable function with bounded partial derivatives and $F=(F_1, \ldots, F_m)$ is a vector of elements of $\mathbb{D}^{1,q}$ for some $q$, then $\phi(F)\in \mathbb{D}^{1,q}$ and
\begin{equation}
D \phi(F) = \sum_{i=1}^{m} \frac{\partial \phi}{\partial x_i} (F) \, D F_i. \label{eq:ChainRule}
\end{equation}
We remark that the conditions on $\phi$ are not optimal and can be weakened.

For $F \in L^2(\Omega)$, with chaotic expansion as in \eqref{eq:ChaoticExpansion}, we define the \textit{pseudo-inverse of the infinitesimal generator of the Ornstein-Uhlenbeck semigroup} as
\[ L^{-1} F = - \sum_{p=1}^{\infty} \frac{1}{p} I_p(f_p). \]

The following \textit{integration-by-parts formula} is one of the main ingredients in Section \ref{sec:proof-Upper-Bound} for proving the upper bound in the main theorem \ref{thm:Main-thm}. It says that for random elements $F,G \in \mathbb{D}^{1,2}$,
\begin{equation}
\E[FG] = \E[F] \E[G] + \E[ \sprod{DG, -DL^{-1}F}_{\HH} ].  \label{eq:IntegrationByParts}
\end{equation}

\subsubsection{Gamma operators and cumulants}\label{sec:cumulant}

Let $F$ be a random variable with characteristic function $\phi_F(t) = \E[ e^{itF}]$, where $i$ stands for the imaginary unit. Its \textit{$j$-th cumulant}, $j\in\N$, denoted by  $\kappa_j(F)$, is defined as
\[ \kappa_j(F) = \frac{1}{i^j} \frac{\partial^j}{\partial t^j} \log \phi_F(t) \Big\vert_{t=0}. \]
Now, let $F$ be a random variable with a finite chaos expansion as at \eqref{eq:ChaoticExpansion}. We define the \textit{Gamma operators} $\Gamma_j$, $j \in \N_0$ recursively via
\begin{equation}
\Gamma_0(F) := F\qquad\text{and}\qquad\Gamma_{j+1} (F) := \sprod{D \Gamma_{j}(F) , -D L^{-1} F}_{\mathfrak{H}}, \quad \text{for } j\geq 0. \label{eq:GammOperatorDefinition}
\end{equation}
We also introduce the centered versions of the Gamma operators by
\[  \CenteredGamma_j(F)  := \Gamma_j(F) - \E[\Gamma_j(F)] .\]

It is important to note that there is an alternative definition of the Gamma operators, which can be found in most other papers in this framework, see for example Definition 8.4.1 in \cite{n-pe-1} or Definition 3.6 in \cite{OptBerryEsseenRates}. For the sake of completeness, we also mention these classical Gamma operators, which we also call \textit{alternative} Gamma operators, which we shall denote by $\Gamma_{alt}$. These are defined via
\begin{equation}
\Gamma_{alt, 0}(F) := F \quad \text{and} \quad \Gamma_{alt, j+1} (F) := \sprod{D F , -D L^{-1} \Gamma_{alt, j}(F)}_{\mathfrak{H}}, \quad \text{for } j\geq 0. \label{eq:GammaOperatorAltDefinition}
\end{equation}
The alternative Gamma operators are related to the cumulants of $F$ by the following identity from \cite{CumOnTheWienerSpace}. For all $j \geq 0$, one has that
$ \E[\Gamma_{alt, j}(F)] = \frac{1}{j!} \kappa_{j+1}(F)$.

Clearly, $\Gamma_j(F) = \Gamma_{alt,j}(F)$, for $j=0$, and $j=1$, while for $j\geq 2$ this is no more the case. However, on the second Wiener chaos the identity is preserved as the following result shows, which will frequently be used throughout this text.

\begin{Prop} \cite[Proposition 2.1]{Optimal-Gamma}\label{Prop:RelationOldAndNewGamma}
	Let $F =I_2(f)  \in \mathscr{H}_2$ for some $f \in \HH^{\odot 2}$ be an element of the second Wiener chaos. Then,
		\[ \Gamma_j(F) = \Gamma_{alt,j}(F) \quad \text{for all } j \geq 0. \]
\end{Prop}

\subsubsection{Useful facts on random elements in the second Wiener chaos}\label{sec:2wiener}

Let $F=I_2(f)$ for some $f \in \mathfrak{H}^{\odot 2}$ be a generic element in the second Wiener chaos. It is a classical result (see \cite[Section 2.7.4]{n-pe-1}) that random variables of this type can be analyzed be means of the associated \textit{Hilbert-Schmidt operator} $A_f : \mathfrak{H} \to \mathfrak{H}$ that maps $g\in\mathfrak{H}$ onto $f \cont{1} g\in\mathfrak{H}$. Denote by $\{ c_{f,i} : i \in \N \}$ the set of eigenvalues of $A_f$. We also introduce the following sequence of auxiliary kernels $( f \contIterated{1}{p} f : p \geq 1 )\subset \HH^{\odot 2}$, defined recursively as $ f \contIterated{1}{1} f = f$, and, for $p \geq 2$ by
$ f \contIterated{1}{p} f = \big( f \contIterated{1}{p-1} f \big) \cont{1} f$.  

\begin{Prop} (see e.g. \cite[p.~43]{n-pe-1})\label{Prop:2choas-properties}
	Let $F=I_2(f)$ with $f \in \mathfrak{H}^{\odot 2}$ be an element from the second Wiener chaos.
	\begin{enumerate}\setlength\itemsep{-0.3em}
		\item[(a)] The random variable $F$ admits the representation \begin{equation} F = \sum_{i=1}^{\infty} c_{f,i} \left( N_i^2 - 1 \right), \label{eq:SecondWienerChaosEigenvalueRepresentation} \end{equation}
		where the $(N_i:i\geq 1)$ are independent standard Gaussian random variables. The random series converges in $L^2(\Omega)$ and almost surely.
		\item[(b)] For every $p \geq 2$, the $p$-th cumulant $\kappa_p(F)$ of $F$ is given by
		\begin{equation}
		\begin{split} 
		\kappa_p(F)  &=2^{p-1} (p-1)! \sum_{i=1}^{\infty} c_{f,i}^p  \\
		&= 2^{p-1} (p-1)!  \langle f, f \contIterated{1}{p-1} f \rangle_{\HH}\\
		&= 2^{p-1} (p-1)! \Tr \left( A^p_f \right), \label{eq:SecondWienerChaosCumulantFormulaEigenvalues} \end{split}\end{equation}
		where $ \Tr ( A^p_f )$ stands for the trace of the $p$-th power of operator $A_f$.
	\end{enumerate}
\end{Prop}

\subsection{Variance--Gamma distributions: basic properties and Stein estimates}

Recall that a random variable $Y$ is said to have a Variance--Gamma (VG) probability distribution with parameters $r >0, \theta \in \R, \sigma >0, \mu \in \R$ if and only if its probability density function $p_{\text{VG}} (x; r, \theta,\sigma,\mu)$ is given by  \eqref{eq:VGdensity}. We write  $Y \sim \VG(r,\theta,\sigma,\mu)$ in this situation. In the limiting case $\sigma \to 0$ the support becomes the open interval $(\mu,\infty)$ if $\theta>0$, and is $(-\infty, \mu)$ if $\theta <0$. Also, it is known that for $Y \sim \text{VG}(r,\theta,\sigma,\mu)$ one has
\begin{equation}\label{eq:mean-variance}
\E(X)=\mu + r \theta, \qquad \text{ and } \qquad \text{Var}(X)=r (\sigma^2 + 2 \theta^2),
\end{equation}
see for example relation (2.3) in \cite{g-variance-gamma}. In this paper, we are interested in centered target distributions, so without loss of generality we set $\mu=-r\theta$, and write $\CVG$ for $\VG (r,\theta,\sigma,-r\theta)$ to denote the centered Variance--Gamma distribution with parameters $r>0$, $\theta\in\R$ and $\sigma>0$.  In particular, when $\theta =0$, the random variable $Y \sim VG_c(r,\theta=0,\sigma)$ has the symmetric centered Variance--Gamma distribution $\text{SVG}_c (r,\sigma)$. It is known that $Y \sim \CVG$ if and only if 
\begin{equation}\label{eq:Center-VG-Representation}
Y \stackrel{\text{law}}{=} \theta (G - r) + \sigma \sqrt{G} N
\end{equation}
where $N \sim N(0,1)$ is standard Gaussian and $G \sim \Gamma(r/2,1/2)$ is Gamma distributed with parameters $r/2$ and $1/2$, and $N$ and $G$ are independent, see \cite[Proposition 3.6]{Robert-Thesis}. 
Moreover, the following formulas from \cite[Lemma 3.6]{VarianceGammaPaper} for the first sixth cumulants of $Y \sim \CVG$ will be used in the proof of Lemma  \ref{lem:Gamma-cumulants}:
\begin{equation}\label{eq:VG-cumulants}
\begin{split}
\kappa_2 (Y)  &=  r (\sigma^2 + 2\theta^2), \\
\kappa_3(Y)&=2r\theta (3\sigma^2+4\theta^2),\\
\kappa_4(Y)&=6r (\sigma^4+8\sigma^2\theta^2+8\theta^4)\\
\kappa_5(Y)&=24 r\theta  (  5\sigma^4 +20 \sigma^2\theta^2 +16 \theta^4  ),\\ \kappa_6(Y)&= 120r (\sigma^2 + 2 \theta^2) ( \sigma^4 + 16\sigma^2 \theta^2 + 16\theta^4).
\end{split}
\end{equation}

Next, we derive a distributional identity for random variables having a centred VG distribution and belong to the second Wiener chaos.

\begin{prop}\label{lem:VG-Intersection-2Chaos}
Let $Y \sim \CVG$. Assume further that $Y$ belongs to the second Wiener chaos. Then, $r \in \N$ is an integer, and there exist $\alpha,\beta>0$ with $2 (\alpha^2 + \beta^2)=\sigma^2+2\theta^2$ such that
\begin{equation}
Y \stackrel{\text{law}}{=} \sum_{i=1}^{r} \alpha (N^2_i-1) - \sum_{j=1}^{r}\beta (\widetilde{N}^2_j-1)
\end{equation} 
where $(N_i, \widetilde{N}_i : \, i \ge 1)$ are independent standard Gaussian random variables. In particular, $ \theta=\alpha - \beta$, and $\sigma=2\sqrt{\alpha \beta}$.
\end{prop}
\begin{proof}
		Let $\phi_Y (t):= \E[e^{itY}]$ be the characteristic function of $Y$. Then, by using a conditioning argument, we obtain 
		\begin{equation}\label{eq:Centered-VG-Characteristic-Function}
		(\phi_Y(t))^{-2}= e^{2it\theta r}\left(  1 - i 2 \theta  t + \sigma^2  t^2 \right)^{r},
		\end{equation}
		see \cite[Equation (7)]{Madan98}.
		On the other hand, there exist $\alpha,\beta>0$ with $\alpha-\beta = \theta$, and $\sigma^2= 4\alpha\beta$ so that we can write $ 1 - i 2 \theta  t + \sigma^2  t^2= (1-2i\alpha t) (1+2i\beta t)$. 
		Next, using the assumption that $Y$ is an element in the second Wiener chaos, the characteristic function of $Y$ can be also expressed as 
		\begin{equation}\label{eq:Characteristic-Function-2Chaos}
		(\phi_Y(t))^{-2}= \prod_{k\ge1} e^{2it\alpha_k} \left( 1-2i\alpha_k t \right)
		\end{equation}
		where the spectral coefficients $(\alpha_k  \, :  \, k \in \N)$ satisfy $\sum_{k \ge 1} \alpha^2_k < \infty$. 
		Now, by comparison \eqref{eq:Centered-VG-Characteristic-Function} and \eqref{eq:Characteristic-Function-2Chaos}, and a root argument, one can infer that $\alpha_k = \alpha$ or $-\beta$ for all $ k\ge 1$. Hence, the claim follows from a standard comparison between characteristic functions. 
	\end{proof}

Next, we turn to Stein's method for VG approximation.  Following \cite{Robert-Thesis, VarianceGammaPaper}, a Stein equation for the centered Variance--Gamma distribution $\CVG$ associated with a test function $h:\R \to \R$ is given by the following second order ordinary differential equation:
\begin{equation}\label{eq:Stein-CVG}
\sigma^2 (x+r\theta) f''(x) +  (\sigma^2 r + 2 \theta(x + r \theta))f'(x) - x f(x) = h(x) - \CVG (h),
\end{equation}
where $\CVG (h)= \E[h(Y)]$ and $Y \sim \CVG$. Moreover, we assume that $\E \vert h(Y) \vert < + \infty$. In \cite{g-variance-gamma} it was shown that a solution to \eqref{eq:Stein-CVG} is given by
\begin{equation}
\begin{aligned}\label{eq:stein-vg-solution-integralform}
f_h (x)  &  =  - \frac{e^{- \beta x} K_\nu   \left(   \alpha  \vert x \vert  \right)}{\sigma^2 \vert x \vert^\nu}    \,  \int_{0}^{x}   e^{\beta y} \vert y \vert^\nu  I_\nu    \left(    \alpha  \vert y \vert \right)  \widetilde{h}(y) dy \\
& \qquad\qquad\qquad  -   \frac{e^{-\beta x}  I_\nu   \left(   \alpha \vert x \vert  \right)}{\sigma^2 \vert x   \vert^\nu}    \int_{x}^{\infty }    e^{\beta y} \vert y \vert^\nu  K_\nu    \left(   \alpha  \vert y \vert \right)  \widetilde{h}(y) dy \\
& = - \frac{e^{- \beta x} K_\nu   \left(   \alpha  \vert x \vert  \right)}{\sigma^2 \vert x \vert^\nu}    \,  \int_{0}^{x}   e^{\beta y} \vert y \vert^\nu  I_\nu    \left(    \alpha  \vert y \vert \right)  \widetilde{h}(y) dy \\
& \qquad\qquad\qquad  +  \frac{e^{-\beta x}  I_\nu   \left(   \alpha \vert x \vert  \right)}{\sigma^2 \vert x   \vert^\nu}    \int_{-\infty}^{x }    e^{\beta y} \vert y \vert^\nu  K_\nu    \left(   \alpha  \vert y \vert \right)  \widetilde{h}(y) dy,
\end{aligned}
\end{equation}
where $\widetilde{h} = h - \CVG (h)$, $\nu = \frac{r-1}{2}$, $\alpha = \frac{\sqrt{\theta^2 + \sigma^2}}{\sigma^2}$, $\beta = \frac{\theta}{\sigma^2}$, and $I_\nu$ represents the modified Bessel function of the first kind. Also, if $h$ is bounded, then $f_h (x)$ and the derivative $f'_h (x)$ are bounded for all $x \in \R$, and \eqref{eq:stein-vg-solution-integralform} is the unique bounded solution when $r \ge 1$, and the unique  solution with bounded first derivative if $r >0$. The next proposition plays a significant role in Section \ref{sec:operator-theory} and gathers the essential ingredients for our purposes on the regularity of the solution of the VG Stein equation \eqref{eq:Stein-CVG}.

\begin{prop}[\cite{gaunt-VG-perfect-bounds}]\label{prop:Stein-Solution-Properties} 
Let  $f_h$ denote the solution  \eqref{eq:stein-vg-solution-integralform} of the Variance--Gamma Stein equation  \eqref{eq:Stein-CVG} associated with the test function $h$.
\begin{itemize}\setlength\itemsep{-0.3em}
\item[(a)]  Assume that $h:\R  \to \R$ is Lipschitz. Then
\begin{equation*}
\begin{aligned}
\Vert    f_h  \Vert_\infty   \le   D_0 (r,\theta,\sigma)  \Vert  h' \Vert_\infty,\\
 \Vert    f'_h  \Vert_\infty   \le   D_1 (r,\theta,\sigma)  \Vert  h' \Vert_\infty,\\
 \Vert    f''_h  \Vert_\infty   \le   D_2 (r,\theta,\sigma)  \Vert  h' \Vert_\infty,
\end{aligned}
\end{equation*}
where the constants $D_i(r,\theta,\sigma)>0$, $i\in\{0,1,2\}$, only depend on $r$, $\theta$ and $\sigma$, and are explicitly given in Equations (3.15)-(3.18) in \cite{gaunt-VG-perfect-bounds}.
\item[(b)] Assume furthermore that the function $h:\R \to \R$ is such that its first derivative $h'$ is bounded and Lipschitz. Then,
\begin{equation*}
\Vert   f^{(3)}_h \Vert_\infty   \le D_3 (r,\theta,\sigma) \Big\{   \Vert h' \Vert_\infty  +  \Vert h''\Vert_\infty \Big\},
\end{equation*}
where the constant $D_3(r,\theta,\sigma)>0$ only depends on $r$, $\theta$ and $\sigma$, and is explicitly given in \cite[Corollary 3.2]{gaunt-VG-perfect-bounds}.
\end{itemize}
\end{prop}

In the proof of Proposition \ref{prop:S-Compact} below we need the following technical result.

\begin{prop}\label{prop:equivanishing-property}
The solution $f_h$ given by \eqref{eq:stein-vg-solution-integralform} of the Variance--Gamma Stein equation  \eqref{eq:Stein-CVG} associated with a bounded test function $h$ satisfies the following properties.
\begin{itemize}\setlength\itemsep{-0.3em}
	\item[(a)] There exists a function $U:\R\to \R$ with $\vert U(x)\vert \to 0$, as $\vert x \vert \to \infty$, so that  $\vert f_h (x) \vert  \le  \,  \Vert h\Vert_\infty \,  U(x)$.
	\item[(b)] There exists a function $V:\R\to \R$ with $\vert V(x)\vert \to 0$, as $\vert x \vert \to \infty$, so that  $\vert f'_h (x) \vert  \le  \,  \Vert h\Vert_\infty \,  V(x)$.
	\end{itemize}
\end{prop}
\begin{proof} 
Both statements are direct consequences of the estimates $(3.12)$, $(3.13)$ given in \cite[Theorem 3.1]{gaunt-VG-perfect-bounds}.
\end{proof}

\subsection{Malliavin--Stein method for Variance--Gamma approximation}\label{sec:MalliavinStein-VG}

In this section we recall some elements related to the Malliavin-Stein method for Variance--Gamma approximation originally developed in \cite{VarianceGammaPaper} and later refined in \cite{gaunt-VG-perfect-bounds}. We start with the following useful observation that relates the variance of a linear combination of the iterated Gamma operators introduced via relation  \eqref{eq:GammOperatorDefinition} to that of a linear combination of cumulants.

\begin{lem}\label{lem:Gamma-cumulants}
	Assume that $F = I_2 (f) \in \mathscr{H}_2$ is an element belonging to the second Wiener chaos. Then, for  $\ell \ge 1$ one has that
\begin{multline}\label{eq:Linear-Combiniation-Cumulants}
	\Var \left(   \Gamma_{\ell+1} (F) - 2 \theta \Gamma_{\ell} (F) - \sigma^2 \Gamma_{\ell-1}(F) \right)
	\\= \frac{\kappa_{2\ell+4}(F)}{(2\ell+3)!} - 4 \theta \frac{\kappa_{2\ell+3}(F)}{(2\ell+2)!} + \left( 4 \theta^2 - 2 \sigma^2  \right) \frac{\kappa_{2\ell+2}(F)}{(2\ell+1)!} + 4 \theta \sigma^2 \frac{\kappa_{2\ell+1}(F)}{(2\ell)!} + \sigma^4 \frac{\kappa_{2\ell} (F)}{(2\ell-1)!}.
	\end{multline}	
	In particular, when $\ell=1$,
	\begin{equation}\label{eq:Gamma_{2,1,0}<M(F)}
	\begin{aligned}
	\Var \Big(  & \Gamma_2 (F) - 2 \theta \Gamma_1 (F) - \sigma^2 F \Big)  \\
	&= \frac{\kappa_6(F)}{5!} - 4 \theta \frac{\kappa_5(F)}{4!} + \left( 4 \theta^2 - 2 \sigma^2  \right) \frac{\kappa_4(F)}{3!} + 4 \theta \sigma^2 \frac{\kappa_3(F)}{2!} + \sigma^4 \kappa_2 (F).
	\end{aligned}
	\end{equation}
Furthermore, when $F \stackrel{\text{law}}{=} Y \sim \CVG$ has a centered Variance--Gamma distribution, then 
	\begin{equation}\label{eq:Linear-Cumulants-2-6} \frac{\kappa_6(Y)}{5!} - 4 \theta \frac{\kappa_5(Y)}{4!} + \left( 4 \theta^2 - 2 \sigma^2  \right) \frac{\kappa_4(Y)}{3!} + 4 \theta \sigma^2 \frac{\kappa_3(Y)}{2!} + \sigma^4 \kappa_2 (Y) = 0.
	\end{equation}
\end{lem}

\begin{proof}
To prove part (a) we use relation \cite[Equation (24)]{a-p-p}, saying that
\begin{equation}\label{eq:CentredGammaInTermsOfContraction}
\CenteredGamma_\ell(F) = \Gamma_\ell (F) - \E[\Gamma_\ell (F)] =  2^\ell I_2 \big(f \contIterated{1}{\ell+1} f \big),
\end{equation}
and the isometry property \eqref{eq:IsometryProperty}. This allows us to conclude that
\begin{align*}
&\Var \left(   \Gamma_{\ell+1} (F) - 2 \theta \Gamma_{\ell} (F) - \sigma^2 \Gamma_{\ell-1}(F) \right)\\
& =  2 \, \Big \Vert  2^{\ell +1} f \contIterated{1}{\ell +2} f - 2^{\ell +1} \theta  f \contIterated{1}{\ell+1} f - 2^{\ell -1} \sigma^2 f \contIterated{1}{\ell} f    \Big\Vert^2_{\HH}\\
&= 2^{2\ell+3}  \langle f, f \contIterated{1}{2\ell+3} f\rangle_{\HH} - 4 \theta 2^{2\ell+2}  \langle f, f \contIterated{1}{2\ell+2} f\rangle_{\HH} + (4\theta^2 - 2 \sigma^2) 2^{2\ell+1}  \langle f, f \contIterated{1}{2\ell+1} f\rangle_{\HH}\\
&\qquad\qquad\qquad+ 4 \theta \sigma^2  2^{2\ell}  \langle f, f \contIterated{1}{2\ell}  f\rangle_{\HH} + \sigma^4 2^{2\ell-1}  \langle f, f \contIterated{1}{2\ell-1}  f\rangle_{\HH}.
\end{align*}
Now, the result follows by using item (b) in Proposition \ref{Prop:2choas-properties}. Finally, the identity \eqref{eq:Linear-Cumulants-2-6}  follows by a direct computation via the cumulant relations \eqref{eq:VG-cumulants}.
\end{proof}

Next, we rephrase, in  a slightly different form, the result obtained in \cite[Theorem 4.1]{VarianceGammaPaper}, which is the starting point for our analysis. We accentuate the recent development on Stein's method for VG approximation in \cite{gaunt-VG-perfect-bounds} that permits us to state the result in the stronger $d_{\mathcal{H}_1}$-distance instead of the so-called bounded Wasserstein distance $d_{\mathcal{H}_{b,1}}$ and with fully explicit constants. 

\begin{thm}[\cite{VarianceGammaPaper,gaunt-VG-perfect-bounds}]\label{thm:MS-CVG-General-Upper-Bound}
Let $Y \sim VG_c(r,\theta,\sigma)$ be a centered Variance--Gamma random variable with parameters $r>0, \theta\in\R$, and $\sigma>0$. 
\begin{itemize}
\item[(a)] Let $F$ be a centered random variable admitting a finite chaos expansion with $\E[F^2]=\E[Y^2]$. Then 
\begin{equation}\label{eq:Main-General-Estimate:CP}
\begin{aligned}
d_{\mathcal{H}_{1}} (F,Y) & \le  C_1 \E \Big \vert  \CenteredGamma_{alt,2}(F) - 2 \theta \CenteredGamma_{alt,1}(F) - \sigma^2 F \Big \vert  + C_2 \Big \vert  \kappa_3 (F) - \kappa_3 (Y) \Big \vert \\
& \le  C_1  \sqrt{  \Var \left(   \Gamma_{alt,2} (F) - 2 \theta \Gamma_{alt,1} (F) - \sigma^2 F \right)  } + C_2 \Big \vert  \kappa_3 (F) - \kappa_3 (Y) \Big \vert,
\end{aligned}
\end{equation}
where  
\begin{align}\label{eq:MS-CVG-constants}
C_1  &  = \frac{1}{\sigma^2} \Big\{ \frac{2}{r+2} A_{r+1,\theta,\sigma}   \Big \} \Big\{    1+ \left(  2+\frac{\theta^2}{\sigma^2} B_{r,\theta,\sigma} \right)    \Big \},  \qquad   C_2 = \frac{1}{2} C_1,\\
B_{r,\theta,\sigma} &= 6 + \frac{2 \sqrt{2}}{\sqrt{r}} + 2 \sqrt{2\pi(r+1)} \frac{\vert \theta \vert}{\sigma} \left(  1+ \frac{\theta^2}{\sigma^2} \right)^{\frac{r-1}{2}}  +2 ( \sqrt{2r} + r)A_{r,\theta,\sigma},\\
A_{r,\theta,\sigma}  &=  
\begin{cases}
\frac{2\sqrt{\pi}}{\sqrt{2r-1}}  \left(  1+ \frac{\theta^2}{\sigma^2} \right)^{\frac{r}{2}},   \mbox{  if }   \,  r \ge 2,\\
12 \Gamma(\frac{r}{2}) \left(   1+ \frac{\theta^2}{\sigma^2}  \right), \mbox{ if } \,   r \in (0,2).
\end{cases} 
\end{align}
\item[(b)] Suppose that $F=I_2 (f)\in \mathscr{H}_2$ belongs to the second Wiener chaos and satisfies $\E[F^2]=\E[Y^2]$. Then 
\begin{equation}\label{eq:2chaos-SquareRoot}
\begin{aligned}
d_{\mathcal{H}_{1}} (F,Y) &  \le C_1 \Bigg\{   \frac{1}{\sqrt{5!}} \, \sqrt{  \big\vert     \kappa_6(F) - \kappa_6 (Y)  \big \vert  }  + 2  \sqrt{  \frac{\vert \theta \vert}{4!}} \, \sqrt{  \big \vert   \kappa_5(F) - \kappa_5(Y) \big \vert }\\
&  \qquad  \qquad   + \sqrt{  \frac{  \vert  4\theta^2 - 2 \sigma^2  \vert  }{3!}} \, \sqrt{    \big \vert     \kappa_4(F) - \kappa_4(Y)  \big \vert }  + \sigma \sqrt{2 \vert   \theta \vert} \,  \sqrt{     \big   \vert    \kappa_3(F)    - \kappa_3(Y)  \big \vert}  \Bigg\} \\
&   \qquad +   \frac{C_1}{2} \Big \vert  \kappa_3 (F) - \kappa_3 (Y) \Big \vert, 
\end{aligned}
\end{equation}
where the constant $C_1$ is the same as in (a).  The bound \eqref{eq:2chaos-SquareRoot} can further be simplified to
\begin{equation}\label{eq:2chaos-SquareRoot-clean}
 d_{\mathcal{H}_1}(F,Y) \le C \, \sqrt{\mathbf{M}(F)}
 \end{equation}
  where $C=C_1 \max\{ \frac{1}{2}, 2  \sqrt{  \frac{\vert \theta \vert}{4!}}, \sqrt{  \frac{  \vert  4\theta^2 - 2 \sigma^2  \vert  }{3!}}, \sigma \sqrt{2 \vert   \theta \vert} \}$ and, where we recall that the quantity $\mathbf{M}(F)$ is given by \eqref{eq:M(F)}.
\end{itemize}
\end{thm}
\begin{proof} The general estimate \eqref{eq:Main-General-Estimate:CP} can be achieved via the Stein equation \eqref{eq:Stein-CVG}, the Malliavin integration-by-parts formula  \eqref{eq:IntegrationByParts}, along with the universal Stein bounds at item (a) in Proposition \ref{prop:Stein-Solution-Properties}. The reader is referred to \cite[Theorem 4.1]{VarianceGammaPaper} for details. The estimate \eqref{eq:2chaos-SquareRoot} is a direct consequence of Proposition \ref{Prop:RelationOldAndNewGamma} and Lemma \ref{lem:Gamma-cumulants}.
\end{proof}

\begin{rem}
	\begin{itemize}	\setlength\itemsep{-0.3em}
		\item[(a)] The assumption in part (a) of Theorem \ref{thm:MS-CVG-General-Upper-Bound} that $F$ admits a finite chaotic expansion can be relaxed. However, this direction is not the focus of the present paper.
	\item[(b)] Imposing the assumption that $\E[F^2]=\E[Y^2]$ in Theorem \ref{thm:MS-CVG-General-Upper-Bound} is no restriction of generality. In fact, if it is not satisfied, one can work with the bound
\begin{equation*}
\begin{split}
d_{\mathcal{H}_{1}} (F,Y)  & \le  C_1 \sqrt{  \Var \left(   \Gamma_{alt,2} (F) - 2 \theta \Gamma_{alt,1} (F) - \sigma^2 F \right)  } \\
 & \qquad + C_2 \Big \vert  \kappa_3 (F) - \kappa_3 (Y) \Big \vert  + C_3 \Big \vert  \kappa_2 (F) - \kappa_2 (Y) \Big \vert
 \end{split}
\end{equation*}
in which $C_1,C_2,C_3>0$ are suitable explicit constants.
\item[(b)] When the target random variable $Y \sim \text{SVG}_c(r, \theta=0,\sigma)$ is symmetric, a closer look at the bound \eqref{eq:2chaos-SquareRoot} reveals that one can rewrite \eqref{eq:2chaos-SquareRoot-clean} as 
\begin{equation}\label{eq:eq:2chaos-SquareRoot-clean-symmetric}
d_{\mathcal{H}_{1}} (F,Y)   \le  C  \sqrt{  \mathbf{M}^{\prime} (F)}
\end{equation}
where the quantity $\mathbf{M}^{\prime} (F)$ is given in Remark \ref{rem:main-thm}, item (a).  Again, the presence of the third cumulant difference $|\kappa_3(F)-\kappa_3(Y)|$ is inevitable, see \cite[Example 2.1]{a-a-p-s}.
\end{itemize}
\end{rem}

\section{Proofs}

\subsection{Proof of Theorem  \ref{thm:Main-thm}}\label{sec:proof-Upper-Bound}

\subsubsection{Variance estimates}

The next two propositions provide the auxiliary estimates towards the optimal upper bound in terms of the variance of the iterated Gamma operators of Malliavin calculus.

\begin{prop}\label{prop:Variance-Estimate-1}
	Let $Y \sim \CVG$ and let $F=I_2 (f) \in \mathscr{H}_2$ be a random variable belonging to the second Wiener chaos such that $\E[F^2]=\E[Y^2]=r (\sigma^2 + 2 \theta^2)$. Put $C=C(r,\theta,\sigma) = 2r (\sigma^2 + 2 \theta^2)$. Then, for $\ell \ge 1$,
	\begin{equation}\label{eq:Variance-Estimate-1}
	\begin{split}
	\Var \left(   \Gamma_{\ell+2} (F) - 2 \theta \Gamma_{\ell+1} (F) - \sigma^2 \Gamma_\ell(F) \right) & \le C(r,\theta,\sigma) \,   \Var \left(   \Gamma_{\ell+1} (F) - 2 \theta \Gamma_{\ell} (F) - \sigma^2 \Gamma_{\ell-1}(F) \right)\\
	& \le C(r,\theta,\sigma)^\ell  \,  \Var \left(   \Gamma_2 (F) - 2 \theta \Gamma_1 (F) - \sigma^2 F \right).
	\end{split}
	\end{equation}	
\end{prop}

\begin{proof}
	Using the representation \eqref{eq:CentredGammaInTermsOfContraction} and the isometry  \eqref{eq:IsometryProperty}, we get
	\begin{align*}
	&\Var \left(   \Gamma_{\ell+2} (F) - 2 \theta \Gamma_{\ell+1} (F) - \sigma^2 \Gamma_\ell(F) \right) \\
	&\qquad= 2 \, \Big \Vert    2^{\ell+2}  f \contIterated{1}{\ell+3} f - 2^{\ell+2} \theta f \contIterated{1}{\ell+2} f - 2^\ell \sigma^2 f \contIterated{1}{\ell+1} f     \Big \Vert^2_{\HH}    \\
	&\qquad= 2 \, \Big \Vert 2f \otimes_{1} \left(    2^{\ell+1}  f \contIterated{1}{\ell+2} f - 2^{\ell+1} \theta f \contIterated{1}{\ell+1} f - 2^{\ell-1} \sigma^2 f \contIterated{1}{\ell} f  \right) \Big\Vert^2_{\HH}\\
	&\qquad\le  2 \Vert 2f\Vert^2_{\HH} \, \Big \Vert  2^{\ell+1}  f \contIterated{1}{\ell+2} f - 2^{\ell+1} \theta f \contIterated{1}{\ell+1} f - 2^{\ell-1} \sigma^2 f \contIterated{1}{\ell} f    \Big\Vert^2_{\HH},
	\end{align*}	
	where to obtain the last inequality we used the classical estimate $(4.4)$ in \cite[Lemma 4.2]{OptBerryEsseenRates}.
	Now, the result follows by noticing that  $\E[F^2] = 2 \Vert f\Vert^2_{\HH} = \E[Y^2] = r (\sigma^2 + 2 \theta^2)$.
\end{proof}

\begin{rem}
Assume $F \stackrel{\text{law}}{=}Y \sim \CVG$ has a centered Variance--Gamma distribution belonging to the second Wiener chaos. Then Proposition \ref{prop:Variance-Estimate-1} together with the relation \eqref{eq:Linear-Cumulants-2-6} immediately yield the following fact of independent interest that for every $\ell \ge 1$ the linear combination of cumulants appearing on the right-hand side of \eqref{eq:Linear-Combiniation-Cumulants} always vanishes.	
\end{rem}

The next proposition encodes the splitting procedure of a given test function in the Banach space of Lipschitz functions. In particular, inequality \eqref{eq:ESS-VE2} is the key estimate for our approach.

\begin{prop}
	Let $Y \sim \CVG$. Let $F=I_2 (f) \in \mathscr{H}_2$ be a random variable belonging to the second Wiener chaos such that $\E[F^2]=\E[Y^2]=r (\sigma^2 + 2 \theta^2)$. Then, for $\ell \ge 1$,
	\begin{multline*}
	\Var \Bigg(   \left(  \Gamma_{2\ell+3} (F) - 2 \theta \Gamma_{2\ell+2} (F) - \sigma^2 \Gamma_{2\ell+1}(F)  \right) - 2 \theta   \left(  \Gamma_{2\ell+2} (F) - 2 \theta \Gamma_{2\ell+1} (F) - \sigma^2 \Gamma_{2\ell}(F)  \right) \\
	- \sigma^2   \left(  \Gamma_{2\ell+1} (F) - 2 \theta \Gamma_{2\ell} (F) - \sigma^2 \Gamma_{2\ell-1}(F)  \right) \Bigg)\\
	\le 2 \, \Var^{\,2}  \left(   \Gamma_{\ell+1} (F) - 2 \theta \Gamma_{\ell} (F) - \sigma^2 \Gamma_{\ell-1}(F) \right).
	\end{multline*}	
	In particular for $\ell=1$ one has that
	\begin{multline}\label{eq:ESS-VE2}
	\Var \Bigg(   \left(  \Gamma_{5} (F) - 2 \theta \Gamma_{4} (F) - \sigma^2 \Gamma_{3}(F)  \right) - 2 \theta   \left(  \Gamma_{4} (F) - 2 \theta \Gamma_{3} (F) - \sigma^2 \Gamma_{2}(F)  \right) \\
	- \sigma^2   \left(  \Gamma_{3} (F) - 2 \theta \Gamma_{2} (F) - \sigma^2 \Gamma_{1}(F)  \right) \Bigg)\\
	\le 2 \, \Var^{\, 2}  \left(   \Gamma_{2} (F) - 2 \theta \Gamma_{1} (F) - \sigma^2 F \right).
	\end{multline}
\end{prop}

\begin{proof}
	Using the relation \eqref{eq:CentredGammaInTermsOfContraction}, the isometry property  \eqref{eq:IsometryProperty} and the classical estimate (4.4) in \cite[Lemma 4.2]{OptBerryEsseenRates}, we see that 
	\begingroup
	\allowdisplaybreaks
	\begin{align*}
	&\Var \Bigg(   \left(  \Gamma_{2\ell+3} (F) - 2 \theta \Gamma_{2\ell+2} (F) - \sigma^2 \Gamma_{2\ell+1}(F)  \right) - 2 \theta   \left(  \Gamma_{2\ell+2} (F) - 2 \theta \Gamma_{2\ell+1} (F) - \sigma^2 \Gamma_{2\ell}(F)  \right) \\
	&\hspace{5cm}- \sigma^2   \left(  \Gamma_{2\ell+1} (F) - 2 \theta \Gamma_{2\ell} (F) - \sigma^2 \Gamma_{2\ell-1}(F)  \right) \Bigg)\\
	&\qquad= 2 \Bigg \Vert       2^{2\ell+3} 	f \contIterated{1}{2\ell+4} f - \theta 2^{2\ell+4} f \contIterated{1}{2\ell+3} f + \left( \theta^2 2^{2\ell+3} - \sigma^2 2^{2\ell+2}  \right) f \contIterated{1}{2\ell+2} f  \\
	&\hspace{5cm}+ \theta \sigma^2 2^{2\ell+2} f \contIterated{1}{2\ell+1} f + \sigma^4 2^{2\ell-1} f \contIterated{1}{2\ell} f \Bigg \Vert^2_{\HH}\\
	&\qquad= 2^3 \Bigg \Vert  \left(  2^{\ell+1} 	f \contIterated{1}{\ell+2} f - \theta 2^{\ell+1} f \contIterated{1}{\ell+1} f  - \sigma^2  2^{\ell-1} f \contIterated{1}{\ell} f \right) \\
	&\hspace{5cm}\otimes_{1} \left(  2^{  \ell+1} 	f \contIterated{1}{\ell+2} f - \theta 2^{\ell+1} f \contIterated{1}{\ell+1} f  - \sigma^2  2^{\ell-1} f \contIterated{1}{\ell} f \right) \Bigg \Vert^2_{\HH}\\
	&\qquad\le 2^3 \Bigg \Vert   2^{\ell+1} 	f \contIterated{1}{\ell+2} f - \theta 2^{\ell+1} f \contIterated{1}{\ell+1} f  - \sigma^2  2^{\ell-1} f \contIterated{1}{\ell} f \Bigg \Vert^4_{\HH}\\
	&\qquad= 2 \, \Var^{\, 2}  \left(   \Gamma_{\ell+1} (F) - 2 \theta \Gamma_{\ell} (F) - \sigma^2 \Gamma_{\ell-1}(F) \right).
	\end{align*}
	\endgroup
This completes the argument.
\end{proof}

\subsubsection{A splitting technique}\label{sec:operator-theory}
The methodology introduced in \cite{n-p-optimal} and Theorem  \ref{thm:MS-CVG-General-Upper-Bound} suggests that in order to get the optimal upper bound, one has to analyze the quantity
\begin{equation}\label{eq:OUB-Term}
\Bigg \vert   \E \left[ h(F) \left( \CenteredGamma_{2} (F) - 2 \theta \CenteredGamma_{1}(F) - \sigma^2 F  \right)  \right]	\Bigg \vert
\end{equation}
for a given test function $h:\R \to \R$, which is bounded and Lipschitz. This will be carried out by means of a so-called splitting technique suggested by the crucial variance estimate \eqref{eq:ESS-VE2}. To this end, we adapt the language of operator theory that is employed for first time in \cite{Optimal-Gamma}. We start by introducing the Banach space of Lipschitz functions $(\mathcal{B}, \Vert \cdot \Vert_{\mathcal{B}})$ by
\begin{eqnarray*}
	\mathcal{B}  :=  \{ h :\ R \to \R \, : \, h \text{ Lipschitz}, \,  \Vert h \Vert_\infty <  \infty, \,  \Vert h' \Vert_\infty < \infty \},\qquad
	\Vert h \Vert_{\mathcal{B}}  =  \Vert h \Vert_\infty  +  \Vert h' \Vert_\infty.
\end{eqnarray*}

\begin{prop}\label{prop:S-Bounded-Linear}
	Let $h \in \mathcal{B}$. Denote by $S(h)$ the unique bounded solution of the centered Variance--Gamma  Stein equation \eqref{eq:Stein-CVG} with bounded first derivative. Then $S(h) \in \mathcal{B}$, and moreover the mapping $ S: \mathcal{B} \to \mathcal{B}$ is a bounded linear operator. 
\end{prop}

\begin{rem}
We remark that existence and uniqueness of $S(h)$ in Proposition \ref{prop:S-Bounded-Linear} is guaranteed by \cite[Lemma 3.13, Lemma 3.14]{Robert-Thesis}.
\end{rem}

\begin{proof}[Proof of Proposition \ref{prop:S-Bounded-Linear}]
	Let $h \in \mathcal{B}$. Linearity of $S$ follows directly from  \eqref{eq:Stein-CVG}, and the uniqueness of the bounded solution with bounded first derivative.  Also, $S(h) \in \mathcal{B}$ in virtue of Proposition \ref{prop:Stein-Solution-Properties}.  Finally,
	\begin{multline*} 
	\Vert S(h) \Vert_{\mathcal{B}}=  \Vert S(h) \Vert_\infty + \Vert S'(h)\Vert_\infty \le (D_0+D_1) \Vert h'\Vert_\infty  \le (D_0 + D_1) \left( \Vert h\Vert_\infty + \Vert h' \Vert_\infty \right)
	= (D_0 + D_1) \Vert h \Vert_{\mathcal{B}}, 
	\end{multline*}	
where the constants $D_0$ and $D_1$ are the same as in Proposition \ref{prop:Stein-Solution-Properties}.  Hence, $\Vert S \Vert \le  D_0 + D_1$, and therefore $S$ is a bounded linear operator on $\mathcal{B}$. 
\end{proof}

\begin{prop}\label{prop:S-No-Eigenvalue}
	Let $a,b,c\in \R$ be not equal to zero simultaneously. Consider the bounded linear operator $S: \mathcal{B} \to \mathcal{B}$ defined as in Proposition \ref{prop:S-Bounded-Linear}. Set $L=L(a,b,c):=aS+bS^2+cS^3$, where $S^2$ and $S^3$ stand for the two- and threefold composition of the operator $S$, respectively. Then $L$ is a bounded linear operator and the following statements are in order.
	\begin{enumerate}
		\item[(a)] The operator $L$ does not admit any non-zero eigenvalue, i.e., if $L(h) = \lambda h$ for some non-zero constant $\lambda \in \R$, then necessary $h =0$.
		\item[(b)] For every non-zero scalar $\lambda \in \R$, the operator $I+\lambda L : \mathcal{B} \to \mathcal{B}$ is a one-to-one map, where $I:\mathcal{B} \to \mathcal{B}$ stands for the identity operator.   
	\end{enumerate}
\end{prop}

\begin{proof}
	That $L$ is a bounded linear operator directly follows from the observation that $\Vert L \Vert \le \vert a \vert \Vert S \Vert + \vert b \vert \Vert S \Vert^2 + \vert c \vert \Vert S \Vert^3$. To prove (a), by virtue of the spectral mapping theorem \cite[Theorem 10.33]{Rudin-Functional-Analysis}, it is enough to show that the point spectrum $\sigma_p(S)$ satisfies $\sigma_p(S) = \emptyset$. We proceed by contradiction. Assume that there exists $h \in \mathcal{B}$ such that for some non-zero scalar $\lambda$, 
	\begin{equation}\label{eq:No-EigenValue}
	S(h)=\lambda h 
	\end{equation} 
Let $Y \sim VG_c(r,\theta,\sigma)$. This implies that $\lambda \E\left[  h(Y)\right]=\E \left[ S(h)(Y) \right]=0$, and hence $\E\left[  h(Y)\right]=0$. Next, relation \eqref{eq:No-EigenValue}, together with the Stein equation \eqref{eq:Stein-CVG} implies that the function $h$ satisfies in the ordinary differential equation 
\begin{equation}\label{eq:ODE-EigenValue}
\lambda \sigma^2 (x+r\theta) h''(x) + \lambda (\sigma^2 r + 2 \theta(x + r \theta))h'(x) - \left(\lambda x + 1\right) h(x)=0.
\end{equation}
According to \cite{HandBook-Exact-Solutions-ODE} the general solution of \eqref{eq:ODE-EigenValue} is given by
\begin{equation}\label{eq:2orderODE-GeneralSolution}
h (x) = e^{     \frac{\sqrt{\theta^2+\sigma^2} -\theta}{\sigma^2} x       } \Big\{ C_1 u_1(x) + C_2 u_2(x) \Big\},
\end{equation} 
where $u_1,u_2$ are two linearly independent solution of the so called confluent hypergeometric  equation $xu''+ (r-x)u - \varkappa u=0$ and the constant $\varkappa=\varkappa(r,\theta,\sigma,\lambda)$ is explicit. Also $C_1,C_2\in\R$ are constants. First note that $\frac{\sqrt{\theta^2+\sigma^2} -\theta}{\sigma^2}>0$. On the other hand, it is known that the confluent hypergeometric  equation $xu''+ (r-x)u - \varkappa u=0$  has a singular point at infinity, see \cite{Vaughn-Special-Functions}, Chapter 5, Appendix B. Hence,  as $x \to +\infty$, the general solution $h$ given by \eqref{eq:2orderODE-GeneralSolution} becomes unbounded, unless $C_1=C_2=0$. If either one of constants $C_1$ or $C_2$ would be non-zero, then function $h$ and therefore function $S(h)$ becomes unbounded, which contradicts the fact that $S(h)$ is bounded. Hence, $C_1=C_2=0$, and therefore $h=0$.

For part (b) assume that $\lambda \neq 0$ is a non-zero scalar. Then the mapping $I+\lambda L : \mathcal{B} \to \mathcal{B}$ is a linear operator. Hence, $I +\lambda L$ is a one-to-one map if and only if $\Ker (I + \lambda L) = {0}$. However, latter property follows directly from part (a).
\end{proof}

\begin{prop}\label{prop:S-Compact}
	The bounded linear operator $S: \mathcal{B} \to \mathcal{B}$ defined in Proposition \ref{prop:S-Bounded-Linear} is a compact operator. Moreover, for  any three scalars $a,b,c \in \R$ the operator $L=L(a,b,c):=aS+bS^2+cS^3$ is compact as well. 
\end{prop}
\begin{proof}
Let $U_{\mathcal{B}}:= \{ h \in \mathcal{B} : \norm{h}_{\mathcal{B}}= \Vert h \Vert_\infty + \Vert h' \Vert_\infty \le  1 \}$ denote the unit ball of the Banach space $\mathcal{B}$. We need to show that the image $S \left( U_{\mathcal{B}} \right)$ of the unit ball is a precompact set in $\mathcal{B}$, or equivalently, that every sequence $(S(h_n):n\ge 1) \subseteq S(U_{\mathcal{B}})$ has a convergent subsequence in the topology of the Banach space $\mathcal{B}$. Following the first step presented in the proof of \cite[Proposition 3.7]{Optimal-Gamma} without loss of generality we can assume that 	there exists an element $h \in U_{\mathcal{B}}$ such that $h_n \to h$ pointwise, as $n\to \infty$. Hereafter, we adapt the second parametrization of the Variance--Gamma distribution from \cite[Denition 3.2]{Robert-Thesis} that can easily transform into our parametrization via \cite[Equation (3.3)]{Robert-Thesis}. We now literally follow  \cite[Chapter 3]{Robert-Thesis} for the solution of the Variance--Gamma Stein equation to verify some desired analytic properties of the solution $S(h)$. Due to the integral representation of the solution $S(h)$ \cite[Lemma 3.14, Equation (3.15)]{Robert-Thesis} (see also \eqref{eq:stein-vg-solution-integralform}), and the derivative $S(h)'$ \cite[Lemma 3.16, Equation (3.16)]{Robert-Thesis} (see also  \cite{gaunt-VG-perfect-bounds}, Equation (3.26)) an application of Lebesgue's dominated convergence theorem implies that as, $n \to \infty$,
\begin{equation*}
 S(h_n) \to S(h), \quad \text{ and }  \quad 
 S(h_n)'  \to S(h)',  \quad  \text{ pointwise }.
\end{equation*}
Hence, we are left to show that the latter convergences hold uniformly too. To this end, again following steps (2), and (3) in the proof of \cite[Proposition 3.7]{Optimal-Gamma}, it is enough to show that family $ \mathscr{S}:=  (S(h_n), S(h), S'(h_n), S'(h)   :   \,   n\ge 1)$ is {equivanishing at infinity}, that is, for every given $\epsilon >0$ there exists a compact interval $K \subset \R$ such that $ \big \vert f(x) \big \vert < \epsilon$ for all $f \in  \mathscr{S}$ and $x \notin K$. However, this is a direct consequence of Proposition \ref{prop:equivanishing-property} together with the fact that $\Vert h_n\Vert_\infty, \Vert h \Vert_\infty \le 1$ for every $n\ge1$. Finally, compactness of the operator $L$ follows directly from the fact that the subset of all compact operators constitutes an ideal of bounded operators, see \cite[3.4.10 Proposition]{Banach-Space-Theory}
\end{proof}

\begin{thm}\label{thm:fredholm-alternative}
	Let $a,b,c\in \R$ be not-zero simultaneously. Consider the bounded linear operator $L=L(a,b,c):=aS+bS^2+cS^3$ as in Proposition \ref{prop:S-No-Eigenvalue}.	Let $\lambda \in \R$ be a non-zero. Then for every $h\in \mathcal{B}$ there exists a unique solution $g \in \mathcal{B}$ to the functional equation 
	\begin{equation}\label{eq:functional-eq}
	h = \left( I + \lambda L \right)(g) = g + a \lambda S(g) + b \lambda S^2(g) +c \lambda S^3(g).
	\end{equation}
\end{thm}
\begin{proof}
	This is a direct application of Propositions \ref{prop:S-No-Eigenvalue}, and  \ref{prop:S-Compact}, and the classical Fredholm alternative  \cite[Theorem 3.4.24, p.\ 329]{Banach-Space-Theory}.	
\end{proof}

\begin{prop}\label{prop:Universal-Bound-Fredholm}
For $r > 0$ let $U_{\mathcal{B}} (r): = \{ h \in \mathcal{B} :  \Vert h \Vert_{\mathcal{B}} \le r \}$ denote the centered ball of radius $r$ in $\mathcal{B}$.	Let $a,b,c\in \R$ be  not zero simultaneously and consider the bounded linear operator $L=L(a,b,c):=aS+bS^2+cS^3$ as in Proposition \ref{prop:S-No-Eigenvalue}.	Let $r_1 >0$ and $\lambda \in \R$ be non-zero. Then there exists a {universal} constant $r_2>0$ such that for every $h \in U_{\mathcal{B}} (r_1)$ the unique solution $g$ of the functional equation \eqref{eq:functional-eq} satisfies $\Vert g \Vert_{\mathcal{B}} \le r_2$. 
\end{prop}

\begin{rem}
	We remark that the constant $r_2$ in the previous proposition may depend on the parameters $r_1,\lambda, a, b, c, r, \theta$ and $\sigma$, but is universal with respect to the choice of $h$.
\end{rem}

\begin{proof}[Proof of Proposition \ref{prop:Universal-Bound-Fredholm}]
	Combining Proposition \ref{prop:S-No-Eigenvalue} with Theorem \ref{thm:fredholm-alternative} we see that the linear bounded operator $I + \lambda L : \mathcal{B} \to \mathcal{B}$ is a bijection. Hence, the result follows from the inverse mapping theorem \cite[1.6.6 Corollary]{Banach-Space-Theory}.
\end{proof}

After these operator theoretic preparations we turn now to VG approximation of random elements from the second Wiener chaos. The analysis of the quantity $\Psi_\ell(g)$ with $\ell\geq 0$ and $g\in\mathcal{H}_1$ given by \eqref{eq:lifting-up} below is motivated by the Malliavin-Stein bound arising in the course of the proof of Theorem \ref{thm:Main-thm}.

\begin{prop}\label{prop:lifting-up}
	Let $Y \sim \CVG$ and $F=I_2(f)$ be a random element belonging to the second Wiener chaos with $\E[F^2]=\E[Y^2]$. Let $g \in \mathcal{H}_1$. For $\ell \ge 0$, define 
	\begin{equation} \label{eq:Mixed-Gamma}
	\begin{aligned}
	\Gamma_{\ell+2,\ell+1,\ell}(F) &:= \Gamma_{\ell+2}(F)-2\theta \Gamma_{\ell+1}(F) - \sigma^2 \Gamma_\ell(F), 
	 \end{aligned}
	 \end{equation}
 	as well as the centred version
 	$$
 	\CenteredGamma_{\ell+2,\ell+1,\ell}(F):= \Gamma_{\ell+2,\ell+1,\ell}(F) - \E \left[\Gamma_{\ell+2,\ell+1,\ell}(F) \right].
 	$$
	  Then
	  \begin{equation}\label{eq:lifting-up}
		\begin{aligned}
	\Psi_\ell (g):= \E  \Big[ \Big( g(F) - 2\theta S(g)(F) \Big)   \CenteredGamma_{\ell+2,\ell+1,\ell}(F) \Big]	 = \sum_{k=1}^{4} \Psi_{\ell,k}(g),
	\end{aligned}
	\end{equation}
	where the quantities $\Psi_{\ell,k}(g)$, $k\in\{1,2,3,4\}$, are given by
	\begin{align*}
	\Psi_{\ell,1} (g)& :=   - 2\, \E \Big[ S(g)(F)  \CenteredGamma_{\ell+3,\ell+2,\ell+1}(F)\Big],\\ 
\Psi_{\ell,2}(g) &:=  - \E \Big[ S(g)''(F) \CenteredGamma_{\ell+2,\ell+1,\ell}(F) \CenteredGamma_{2,1,0}(F) \Big],\\
	\Psi_{\ell,3}(g) &:= - \E \Big[ S(g)''(F) \CenteredGamma_{\ell+2,\ell+1,\ell}(F) \Big] \times \Bigg\{  \E \Big[ \Gamma_{2,1,0} (F) \Big] - \E \Big[ \Gamma_{2,1,0} (Y) \Big]\Bigg\}, \\
\Psi_{\ell,4}(g) &:=   - \E \Big[ S(g)(F) \Big] \times\Bigg\{ \E \Big[ \Gamma_{\ell+3,\ell+2,\ell+1}(F) \Big] - \E \Big[ \Gamma_{\ell+3,\ell+2,\ell+1}(Y) \Big]\Bigg\}.
			\end{align*} 
Furthermore, there exists a constant $C>0$ only depending only on $r, \sigma, \theta$ such that  
\begin{equation}\label{eq:r,k=2,3,4}
\Big \vert  \Psi_{\ell,k} (g) \Big \vert \le C \,  \mathbf{M}(F)
\end{equation}
for $k\in\{2,3,4\}$.
	\end{prop}

\begin{proof}
Let $g \in \mathcal{H}_1$. Taking into account that $\E \left[   \CenteredGamma_{\ell+2,\ell+1,\ell}(F) \right]=0$, we can write
\begin{multline}\label{eq:1}
\E \Big[  g(F)   \CenteredGamma_{\ell+2,\ell+1,\ell}(F)  \Big]  = \E  \Big[ \left( g(F) - \E [g(Y)]\right) \CenteredGamma_{\ell+2,\ell+1,\ell}(F) \Big]\\=
\E \Big[ \Big(\sigma^2 (F+r\theta) S(g)''(F) +  (\sigma^2 r + 2 \theta(F + r \theta))S(g)'(F) - F S(g)(F)   \Big)  \CenteredGamma_{\ell+2,\ell+1,\ell}(F) \Big].
\end{multline}
Using Malliavin integration-by-parts formula, we obtain
\begin{align*}
- \E  &\Big[ F S(g)(F)   \CenteredGamma_{\ell+2,\ell+1,\ell}(F) \Big] = -\E \Big[ S(g)(F)  \Gamma_{\ell+3,\ell+2,\ell+1}(F)\Big] - \E \Big[  S(g)'(F)  \CenteredGamma_{\ell+2,\ell+1,\ell}(F) \Gamma_1(F) \Big] .
\end{align*} 
Note that  $\E\left[ \Gamma_1(F) \right]=\kappa_2(F)= r (\sigma^2 + 2 \theta^2)$, and using \eqref{eq:1} we have that
\begin{equation}\label{eq:2}
\begin{aligned}
&- \E \Big[  S(g)'(F)  \CenteredGamma_{\ell+2,\ell+1,\ell}(F) \Big( \CenteredGamma_1(F) - 2 \theta F \Big) \Big] \\
&= -\E \Big[ S(g)''(F) \CenteredGamma_{\ell+2,\ell+1,\ell}(F) \Big( \Gamma_2(F) - 2 \theta \Gamma_1(F) \Big) \Big] \\
&\qquad\qquad- \E  \Big[ S(g)'(F)  \Big \langle  D\CenteredGamma_{\ell+2,\ell+1,\ell}(F), -DL^{-1}(\CenteredGamma_1(F) - 2 \theta F) \Big \rangle_\mathfrak{H}     \Big].
\end{aligned}
\end{equation}
Note that 
\begin{align*}
\E \left[ \Gamma_{2,1,0} (F) \right] &= \frac{1}{2}\kappa_3(F) - 2 \theta \kappa_2(F),\\
\E \left[ \Gamma_{2,1,0} (Y) \right] &= r\theta \sigma^2
\end{align*}
and that, for $\ell \ge 0$, we have 
$$
\E \left[  \Gamma_{\ell+3,\ell+2,\ell+1}(Y) \right]=0.
$$
Plugging these identities into \eqref{eq:1} we obtain 
\begin{equation}\label{eq:3}
\begin{aligned}
&\E\Big[  g(F)   \CenteredGamma_{\ell+2,\ell+1,\ell}(F)  \Big] \\
&\qquad = - \E \Big[ S(g)(F)  \Gamma_{\ell+3,\ell+2,\ell+1}(F)\Big] - \E \Big[ S(g)''(F) \CenteredGamma_{\ell+2,\ell+1,\ell}(F) \Gamma_{2,1,0}(F) \Big]\\
& \hskip2cm -  \E  \Big[ S(g)'(F) \Big \langle  D\CenteredGamma_{\ell+2,\ell+1,\ell}(F), -DL^{-1}(\CenteredGamma_1(F) - 2 \theta F) \Big \rangle_{\HH}     \Big]\\
&\qquad  = - \E \Big[ S(g)(F)  \Gamma_{\ell+3,\ell+2,\ell+1}(F)\Big] - \E \Big[ S(g)''(F) \CenteredGamma_{\ell+2,\ell+1,\ell}(F) \CenteredGamma_{2,1,0}(F) \Big]\\
&  \hskip2cm - \E \Big[ S(g)''(F) \CenteredGamma_{\ell+2,\ell+1,\ell}(F) \Big] \times \Bigg\{  \E \Big[ \Gamma_{2,1,0} (F) \Big] - \E \Big[ \Gamma_{2,1,0} (Y) \Big]\Bigg\} \\
&  \hskip2cm -  \E  \Big[ S(g)'(F)  \Big \langle  D\CenteredGamma_{\ell+2,\ell+1,\ell}(F), -DL^{-1}(\CenteredGamma_1(F) - 2 \theta F) \Big \rangle_{\HH}     \Big]\\
& \qquad = - \E \Big[ S(g)(F)  \CenteredGamma_{\ell+3,\ell+2,\ell+1}(F)\Big] - \E \Big[ S(g)''(F) \CenteredGamma_{\ell+2,\ell+1,\ell}(F) \CenteredGamma_{2,1,0}(F) \Big]\\
&  \hskip2cm -  \E  \Big[ S(g)'(F) \Big \langle  D\CenteredGamma_{\ell+2,\ell+1,\ell}(F), -DL^{-1}(\CenteredGamma_1(F) - 2 \theta F) \Big \rangle_{\HH}     \Big]\\
&  \hskip2cm  - \E \Big[ S(g)''(F) \CenteredGamma_{\ell+2,\ell+1,\ell}(F) \Big] \times \Bigg\{  \E \Big[ \Gamma_{2,1,0} (F) \Big] - \E \Big[ \Gamma_{2,1,0} (Y) \Big]\Bigg\} \\
&  \hskip2cm  - \E \Big[ S(g)(F) \Big] \times\Bigg\{ \E \Big[ \Gamma_{\ell+3,\ell+2,\ell+1}(F) \Big] - \E \Big[ \Gamma_{\ell+3,\ell+2,\ell+1}(Y) \Big]\Bigg\}.
\end{aligned}
\end{equation}
Next, using representation \eqref{eq:CentredGammaInTermsOfContraction} a straightforward computation shows that
\begin{equation}\label{eq:InnerProduct-Representation}
\begin{split}
\Big \langle  D\CenteredGamma_{\ell+2,\ell+1,\ell}(F), -DL^{-1}(\CenteredGamma_1(F) - 2 \theta F) \Big \rangle_{\HH}  &= \CenteredGamma_{\ell+4,\ell+3,\ell+2}(F) - 2 \theta \CenteredGamma_{\ell+3,\ell+2,\ell+1}(F) \\
&\qquad\qquad+ \E \left[ \CenteredGamma_{\ell+2,\ell+2,\ell} \left(   \CenteredGamma_1 (F) - 2 \theta F\right) \right].
\end{split}
\end{equation}
Hence,
\begin{equation}\label{eq:4}
\begin{aligned}
\E & \Big[  S(g)'(F)  \Big \langle  D\CenteredGamma_{\ell+2,\ell+1,\ell}(F), -DL^{-1}(\CenteredGamma_1(F) - 2 \theta F) \Big \rangle_{\HH} \Big]\\
 &= \E \Big[ S(g)'(F) \left( \CenteredGamma_{\ell+4,\ell+3,\ell+2}(F) - 2 \theta \CenteredGamma_{\ell+3,\ell+2,\ell+1}(F) \right)  \Big] \\
 &\qquad\qquad+ \E \left[S(g)'(F) \right] \times \E \left[ \CenteredGamma_{\ell+2,\ell+2,\ell} \left(   \CenteredGamma_1 (F) - 2 \theta F\right) \right]\\
 & = \E \Big[ S(g)(F) \left(   \CenteredGamma_{\ell+3,\ell+2,\ell+1}(F) - 2 \theta \CenteredGamma_{\ell+2,\ell+1,\ell}(F)  \right)  \Big]\\
 &\qquad\qquad+ \E \left[S(g)'(F) \right] \times \E \Big[ \CenteredGamma_{\ell+2,\ell+2,\ell} \left(   \CenteredGamma_1 (F) - 2 \theta F\right) \Big]\\
 & \qquad\qquad  - \E \left[S(g)'(F) \right] \times \E \Big[ \Gamma_{\ell+4,\ell+3,\ell+2}(F) - 2 \theta \Gamma_{\ell+3,\ell+2,\ell+1} \Big]\\
 &= \E \Big[ S(g)(F) \left(   \CenteredGamma_{\ell+3,\ell+2,\ell+1}(F) - 2 \theta \CenteredGamma_{\ell+2,\ell+1,\ell}(F)  \right)  \Big],
\end{aligned}
\end{equation} 
where we also used that
$$
\E  \left[  \CenteredGamma_{\ell+2,\ell+1,\ell} \left(   \CenteredGamma_1 (F) - 2 \theta F\right)   \right] = \E \left[ \Gamma_{\ell+4,\ell+3,\ell+2}(F) - 2 \theta \Gamma_{\ell+3,\ell+2,\ell+1}  \right].
$$
Now, plugging \eqref{eq:4} into \eqref{eq:3} we obtain \eqref{eq:lifting-up}. 

Next, we treat the estimates \eqref{eq:r,k=2,3,4} for $k\in\{2,3,4\}$.  Let $k=2$. Using Proposition \ref{prop:Stein-Solution-Properties}, the Cauchy-Schwartz inequality and Proposition \ref{prop:Variance-Estimate-1} we obtain 
\begin{align*}
	\Big \vert  \Psi_{\ell,2} (g) \Big \vert & \le  \E  \Bigg \vert    S(g)''(F) \CenteredGamma_{\ell+2,\ell+1,\ell}(F) \CenteredGamma_{2,1,0}(F) \Bigg \vert    \le C\,  \E  \Bigg \vert    \CenteredGamma_{\ell+2,\ell+1,\ell}(F) \CenteredGamma_{2,1,0}(F) \Bigg \vert   \\
	&  \le C \,  \sqrt{  \Var \left(  \Gamma_{\ell+2,\ell+1,\ell} \right) } \times \sqrt{ \Var\left( \Gamma_{2,1,0}  \right)} \le C\,   \sqrt{ \Var\left( \Gamma_{2,1,0}  \right)} \times \sqrt{ \Var\left( \Gamma_{2,1,0}  \right)} \\
&	\le C \, \mathbf{M}(F),
	\end{align*}
where we have implicitly used \eqref{eq:Gamma_{2,1,0}<M(F)} to obtain the last inequality. Here and in what follows, $C>0$ stands for a constant whose value might change from line to line.

Now, let $k=3$. First, note that using Propositions \ref{prop:Stein-Solution-Properties} and \ref{prop:Variance-Estimate-1}, we obtain
\begin{align*}
 \Bigg \vert \E \Big[ S(g)''(F) \CenteredGamma_{\ell+2,\ell+1,\ell}(F) \Big] \Bigg \vert\le C\,  \sqrt{   \Var\left(  \Gamma_{\ell+2,\ell+1,\ell}(F)\right)} \le C\, \sqrt{\mathbf{M}(F)}.
 \end{align*}
 Hence,
 \begin{align*}
 \Big \vert \Psi_{\ell,3}(g) \Big \vert &= \Bigg \vert \E \Big[ S(g)''(F) \CenteredGamma_{\ell+2,\ell+1,\ell}(F) \Big] \Bigg \vert   \times \Bigg \vert \Bigg\{  \E \Big[ \Gamma_{2,1,0} (F) \Big] - \E \Big[ \Gamma_{2,1,0} (Y) \Big]\Bigg\} \Bigg\vert\\
 & \le C\,  \sqrt{\mathbf{M}(F)}  \times  \Big \vert  \kappa_3(F) - \kappa_3(Y) \Big \vert  \le C\, \mathbf{M}(F)^{3/2}.
\end{align*}
Finally, let $k=4$. First, using Proposition \ref{prop:Stein-Solution-Properties} we have $\big\vert \E \left[ S(g)(F) \right] \big \vert \le C$. Next, we proceed by an induction argument to show that, for $\ell\ge1$, 
\begin{equation}\label{eq:Induction}
\Bigg \vert \E \Big[ \Gamma_{\ell+3,\ell+2,\ell+1}(F) \Big] - \E \Big[ \Gamma_{\ell+3,\ell+2,\ell+1}(Y) \Big]  \Bigg \vert \le C  \, \mathbf{M}(F);
\end{equation} 
note that the statement holds for $\ell=0$ by \eqref{eq:VG-cumulants} and a simple direct check. To show \eqref{eq:Induction} we start by writing
\begin{align*}
 \E  \left[ \Gamma_{\ell+3,\ell+2,\ell+1}(F)\right] &= \frac{\kappa_{\ell+4}(F)}{(\ell+3)!} - 2 \theta \frac{\kappa_{\ell+3}(F)}{(\ell+2)!} - \sigma^2 \frac{\kappa_{\ell+2}(F)}{(\ell+1)!}\\
  &= \frac{\kappa_{\ell+4}(F)}{(\ell+3)!} -4\theta  \frac{\kappa_{\ell+3}(F)}{(\ell+2)!}+ (4\theta^2 - 2\sigma^2) \frac{\kappa_{\ell+2}(F)}{(\ell+1)!} + 4 \theta \sigma^2 \frac{\kappa_{\ell+1}(F)}{\ell!} + \sigma^4 \frac{\kappa_{\ell}(F)}{(\ell-1)!}\\
  & \qquad\qquad+ \sigma^2  \Bigg\{  \frac{\kappa_{\ell+2}(F)}{(\ell+1)!} - 2\theta \frac{\kappa_{\ell+1}(F)}{\ell!} - \sigma^2 \frac{\kappa_\ell(F)}{(\ell-1)!} \Bigg\}
  \\
  &\qquad\qquad+ 2 \theta \Bigg\{   \frac{\kappa_{\ell+3}(F)}{(\ell+2)!} - 2 \theta \frac{\kappa_{\ell+2}(F)}{(\ell+1)!} - \sigma^2 \frac{\kappa_{\ell+1}(F)}{\ell!}  \Bigg \}\\
  &= \frac{\kappa_{\ell+4}(F)}{(\ell+3)!} -4\theta  \frac{\kappa_{\ell+3}(F)}{(\ell+2)!}+ (4\theta^2 - 2\sigma^2) \frac{\kappa_{\ell+2}(F)}{(\ell+1)!} + 4 \theta \sigma^2 \frac{\kappa_{\ell+1}(F)}{(\ell-1)!}\\
  &\qquad\qquad+2\theta \E \left[ \Gamma_{\ell+2,\ell+1,\ell}(F) \right]+\sigma^2 \E \left[  \Gamma_{\ell+1,\ell,\ell-1} (F)\right].
\end{align*}
The two summands in the last line can be handled by means of the induction hypothesis. For the terms in the first line of the last expression, we have two possibilities. If  $\ell=2s$ for some $s\ge1$, then
\begin{align*}
&\frac{\kappa_{\ell+4}(F)}{(\ell+3)!} -4\theta  \frac{\kappa_{\ell+3}(F)}{(\ell+2)!}+ (4\theta^2 - 2\sigma^2) \frac{\kappa_{\ell+2}(F)}{(\ell+1)!} + 4 \theta \sigma^2 \frac{\kappa_{\ell+1}(F)}{\ell!} + \sigma^4 \frac{\kappa_{\ell}(F)}{(\ell-1)!} \\
&\qquad= \Var \left(  \Gamma_{s+1,s,s-1}(F)\right) \le C\, \Var \left( \Gamma_{2,1,0} \right) \le C\, \mathbf{M}(F). 
\end{align*}
If otherwise $\ell=2s+1$ for some $s\ge1$, then using Proposition \ref{prop:Variance-Estimate-1} we obtain
\begin{align*}
&\Bigg \vert \frac{\kappa_{\ell+4}(F)}{(\ell+3)!} -4\theta  \frac{\kappa_{\ell+3}(F)}{(\ell+2)!}+ (4\theta^2 - 2\sigma^2) \frac{\kappa_{\ell+2}(F)}{(\ell+1)!} + 4 \theta \sigma^2 \frac{\kappa_{\ell+1}(F)}{\ell!} + \sigma^4 \frac{\kappa_{\ell}(F)}{(\ell-1)!}\Bigg \vert  \\
&\qquad= \Bigg \vert\frac{\kappa_{2s+5}(F)}{(2s+4)!} -4\theta  \frac{\kappa_{2s+4}(F)}{(2s+3)!}+ (4\theta^2 - 2\sigma^2) \frac{\kappa_{2s+3}(F)}{(2s+2)!} + 4 \theta \sigma^2 \frac{\kappa_{2s+2}(F)}{(2s+1)!} + \sigma^4 \frac{\kappa_{2s+1}(F)}{(2s)!} \Bigg \vert\\
&\qquad =\Bigg \vert \E \left[  \CenteredGamma_{s+2,s+1,s}(F) \CenteredGamma_{s+1,s,s-1}(F) \right] \Bigg \vert \\
&\qquad\le \sqrt{\Var\left( \Gamma_{s+2,s+1,s}(F) \right)} \times \sqrt{\Var \left( \Gamma_{s+1,s,s-1}(F) \right)}\\
&\qquad\le C\, \sqrt{\Var \left( \Gamma_{2,1,0}(F) \right)} \times  \sqrt{\Var \left( \Gamma_{2,1,0}(F) \right)}\\
&\qquad\le C\, \mathbf{M}(F).
\end{align*}
This completes the argument.
\end{proof}

As a final step in the preparation of the proof of the upper bound in Theorem \ref{thm:Main-thm} we develop a general upper bound for the $d_{\mathcal{H}_2}$-distance between a second Wiener chaos random element and a centred VG-distributed random variable.

\begin{prop}\label{prop:First-General-Estimate}
	Let $Y \sim \CVG$ and $F$ be a random element belonging to the second Wiener chaos with $\E[F^2]=\E[Y^2]$. Then there exists a constant $C=C(r,\theta,\sigma)>0$ only depending on $r$, $\theta$ and $\sigma$ such that	
	\begin{align*}
	d_{\mathcal{H}_2} (F,Y)  &\le C \, \Bigg\{   \sup_{h \in \mathcal{H}_{b,1}} \Bigg \vert   \E \left[ h(F)  \CenteredGamma_{2,1,0} (F)   \right]	\Bigg \vert  + \Big \vert  \kappa_3 (F) - \kappa_3 (Y) \Big \vert \Bigg\}\\
	&\le  C \, \Bigg\{   \sup_{h \in \mathcal{H}_{b,1}} \Bigg \vert   \E \Bigg[ \Big( h(F) - 2 \theta S(h) (F) -4 (\theta^2+\sigma^2) S^2(h)(F) \\
	&\qquad\qquad\qquad+ 8\theta (\theta^2+4\sigma^2) S^3(h)(F) \Big)  \CenteredGamma_{2,1,0} (F) \Bigg]	\Bigg \vert + \Big \vert  \kappa_3 (F) - \kappa_3 (Y) \Big \vert \Bigg\}.\\
	\end{align*}
\end{prop}
\begin{proof}
	Let $h \in \mathcal{H}_{2}$ be an arbitrary test function. Then using the Stein equation \eqref{eq:Stein-CVG},  the Malliavin integration-by-parts formula \eqref{eq:IntegrationByParts}, and Proposition \ref{prop:Stein-Solution-Properties}, we obtain the upper bound
	\begin{align*}
	\Bigg \vert \E\left[  h(F) \right] -\E \left[ h (Y)  \right] \Bigg \vert &\le  \Bigg \vert \E \Big[  S(h)'' (F)   \CenteredGamma_{2,1,0} (F)  \Big] \Bigg \vert +  \Bigg \vert \E \big[ S(h)''(F) \big] \times \Big[  \kappa_3 (F) - \kappa_3 (Y) \Big] \Bigg \vert\\ 
	&\le C\, \Bigg\{ \E \Big[ S(h)''(F) \CenteredGamma_{2,1,0}(F) \Big]  + \Big \vert  \kappa_3 (F) - \kappa_3 (Y) \Big \vert \Bigg \} .
	\end{align*}	
	Now, the claims follow by applying Proposition \ref{prop:Stein-Solution-Properties}, Theorem \ref{thm:fredholm-alternative} with $\lambda=1, a=-2\theta, b=-4(\theta^2+\sigma^2), c=8\theta(\theta^2+4\sigma^2)$ and, finally, Proposition \ref{prop:Universal-Bound-Fredholm}.
\end{proof}

\subsubsection{Proof of Theorem \ref{thm:Main-thm}:  upper bound}

In what follows $C>0$ stands for a generic constant whose value may change from line to line. Also, recall that the mapping $\Psi_r (\cdot)$ is defined via relation \eqref{eq:lifting-up}. We apply several times Proposition \ref{prop:lifting-up} to obtain
\begin{equation}\label{ed:Main-Upper-1}
\begin{aligned}
&\Big(  \Psi_0(h) - 2 \theta \Psi_0 \left( S(h) \right)  \Big)  - 2\theta \Big(  \Psi_0   \left( S(h) \right)  - 2\theta \Psi_0 \left( S^2(h) \right)\Big)\\
&\qquad=4 \Psi_3\left( S^2(h)\right) + R_1\\ 
&\qquad=-8 \, \E \left[  S^3(h) (F) \CenteredGamma_{5,4,3} (F)\right] + \widetilde{R}_1.
\end{aligned}
\end{equation}
Here $R_1$ is some remainder term and the remainder term $\widetilde{R}_1$ only consists of terms satisfying the estimate $\vert \widetilde{R}_1 \vert \le C \mathbf{M}(F)$. Similarly, 
\begin{equation}\label{ed:Main-Upper-2}
\begin{aligned}
\Psi_0 \left( 2\theta S(h) \right) - \Psi_0 \left( 8\theta^2 S^2(h) \right)= 16\,\theta \E \left[ S^3(h) (F) \CenteredGamma_{4,3,2}(F) \right] + R_2
\end{aligned}
\end{equation}
where $R_2$ also satisfies $\vert R_2 \vert \le C \mathbf{M}(F)$. Moreover, 
\begin{equation}\label{eq:Main-Upper-3}
-4 \sigma^2 \Psi_0 \left(  S^2(h)  \right) = 8 \sigma^2\, \E \left[S^3(h)(F) \CenteredGamma_{3,2,1}(F) \right]+R_3
\end{equation}
where once more $\vert R_3 \vert \le C \mathbf{M}(F)$. Summing up \eqref{ed:Main-Upper-1}, \eqref{ed:Main-Upper-2}, \eqref{eq:Main-Upper-3}, and using the linearity of map $\Psi_0 ( \cdot )$ we obtain
\begin{equation}
\begin{aligned}
\Big(  \Psi_0(h) - & 2 \theta \Psi_0 \left( S(h) \right)  \Big)  - 2\theta \Big(  \Psi_0   \left( S(h) \right)  - 2\theta \Psi_0 \left( S^2(h) \right)\Big)+\Psi_0 \left( 2\theta S(h) \right) \\
& \qquad  \qquad \qquad  - \Psi_0 \left( 8\theta^2 S^2(h) \right)-4 \sigma^2 \Psi_0 \left(  S^2(h)  \right)\\
& = \Psi_0 \left(   h - 2 \theta S(h) -4 (\theta^2+\sigma^2) S^2(h) + 8\theta (\theta^2+4\sigma^2) S^3(h)  \right)\\
& = \E \Big[ \Big( h(F) - 2 \theta S(h) (F) -4 (\theta^2+\sigma^2) S^2(h)(F) + 8\theta (\theta^2+4\sigma^2) S^3(h)(F) \Big)  \CenteredGamma_{2,1,0} (F) \Big]\\
& = -8 \, \E \Big[ S^3(h) (F) \Big(  \CenteredGamma_{5,4,3}(F) - 2\theta \CenteredGamma_{4,3,2}(F) -\sigma^2 \CenteredGamma_{3,2,1}(F) \Big)  \Big]+R
\end{aligned}
\end{equation}
where $R$ is some term satisfying the estimate $\vert R \vert \le C \mathbf{M}(F)$. Now, the claim follows by applying first Proposition \ref{prop:First-General-Estimate}, then Proposition \ref{prop:Stein-Solution-Properties}, item (b), to infer that
$$
\Vert S^3(h) \Vert_\infty \le C \big\{ \Vert h' \Vert_\infty +  \Vert h''\Vert_\infty \big\}  \le C,
$$
and finally the Cauchy-Schwartz inequality along with variance estimates \eqref{eq:ESS-VE2} and  \eqref{eq:Gamma_{2,1,0}<M(F)}. This completes the proof of the upper bound in Theorem \ref{thm:Main-thm}.\hfill $\Box$

\subsubsection{Proof  of Theorem \ref{thm:Main-thm}: lower  bound}

We essentially follow the same line of arguments already employed in \cite[Section 4.3]{Optimal-Gamma} to prove the lower bound in Theorem \ref{thm:Main-thm}. To this end, first taking into account the second moment assumption $\E[F^2]=\E[Y^2]$ and using hypercontractivity of the elements in Wiener chaos \cite[Corollary 2.8.14]{n-pe-1}, it is a classical result (see \cite[Chapter $7$]{lukas}) that  there exits a  strip $\Delta = \Delta_{r,\theta,\sigma}:= \{ z \in \mathbb{C} :  \abs{ \operatorname{Im}z } <  \delta \}$ in the complex plane (with $\delta = \delta (r,\theta,\sigma)$ depending on $r$, $\theta$ and $\sigma$ so that the two imaginary roots of the polynomial $1- 2 i\theta t + \sigma^2 t^2$  lie outside the strip) such that the characteristic functions $\phi_{F}$ and $\phi_{Y}$ are analytic inside $\Delta$. Moreover, within the strip $\Delta$, they admit the integral representations
\[
\phi_{F}(z) = \int_\R e^{izx} \mu_F(dx) \quad \text{and} \quad \phi_{Y}(z) = \int_\R e^{izx} \mu_Y (dx),
\]
where $\mu_F$ and $\mu_Y$ stand for the distributions of $F$ and $Y$, respectively. Next, recall that all elements in the second Wiener chaos have finite exponential moments, see \cite[Proposition $2.7.13$, item (iii)]{n-pe-1}. Denote for $\rho>0$ by $\Omega_{\rho, \delta} \subseteq \Delta$ the domain
\[\Omega_{\rho, \delta}: = \Big\{ z=t + i y \in \mathbb{C} \, : \, \abs{ \operatorname{Re}z } < \rho, \abs{ \operatorname{Im}z } < \min \{ \delta, e^{-1}\} \Big\}.
\]
Then, for any $z \in \Omega_{\rho,\delta}$, using with a Fubini's argument, we have that
\begin{align*}
\abs[\Big]{ \phi_{F}(z) - \phi_{Y}(z) } & = \abs[\Big]{ \int_\R  e^{itx - y x} (\mu_F - \mu_Y) (dx) } = \abs[\Big]{ \sum_{ k \ge 0} \frac{(-y)^k}{k!} \int_\R x^k e^{itx} (\mu_F - \mu_Y)(dx) } \\
& \le \sum_{k\ge 0}\frac{e^{-k}}{k!} \abs[\Big]{ \phi^{(k)}_{F}(t)- \phi^{(k)}_{Y} (t) }  \le  \sum_{k\ge 0}\frac{e^{-k}}{k!} \rho^{k+1} d_{\mathcal{H}_2} (F, Y) \\
& = \rho \, e^{\rho e^{-1}} d_{\mathcal{H}_2} (F, Y).
\end{align*}
Hence $ \abs{ \phi_{F}(z)  - \phi_{Y}(z) } \le {C_{\rho}} d_{\mathcal{H}_2} (F, Y)$ for every $z \in \Omega_{\rho,\delta}$ and some constant $C_\rho>0$ only depending on $\rho$. Let $R>0$  be such that the disk $D_R \subseteq \Omega_{\rho,\delta}$ with the origin is the center and the radius is $R$. Now, using the fact that
\[
\frac{1}{\phi^2_{Y}(z)}= e^{ir\theta z} (1-2i\theta z + \sigma^2 z^2)^{r/2} ,\qquad z\in D_R,
\]
one can readily conclude that function $\phi_{Y}$ is bounded away from $0$ on the disk $D_R$. Also, for any $\ell \ge 2$,
\begin{equation}\label{eq:lower-2}
\begin{split}
\abs[\big]{ \kappa_\ell (F) } & \le 2^{\ell-1}(\ell-1)! \sum_{i\ge 1} \abs{ c_{i} }^{\ell}   \le 2^{\ell-1}(\ell-1)! \max_{i} \abs{ c_{i} }^{\ell-2} \sum_{i\ge 1} \abs{ c_{i} }^{2}\\
& \le  2^{\ell-2}(\ell-1)! \sqrt{\E[F^2]}^{\,\ell-2} \times \E(F^2) = 2^{\ell-2}(\ell-1)! (\E[F^2])^{\ell/2}.
\end{split}
\end{equation}
Let $\E[F^2]=\E[Y^2]=r(\sigma^2 + 2 \theta)=: \eta = \eta(r,\theta,\sigma)$. Therefore, for any $z \in D_R$,
\begin{align*}
\abs[\Big]{ \frac{1}{\phi_{F}(z)} } \le \exp \Big\{ \sum_{\ell \ge 2} \frac{ \abs{ \kappa_\ell(F) } }{\ell !} \abs{z}^\ell \Big\} & \le  \exp \Big\{ \sum_{\ell \ge 2} \frac{2^{\ell-2}(\ell-1)! \sqrt{\eta}^{\,\ell}}{\ell !} \abs{z}^\ell \Big\}\\
& \le  \exp \Big\{ \sum_{\ell \ge 2} \frac{2^{\ell-2}(\ell-1)! \sqrt{\eta}^{\,\ell}}{\ell !} R^\ell \Big\}=: C_{R,\eta}< \infty,
\end{align*} 
where $C_{R,\eta}$ is a constant depending on $R$ and $\eta$. Hence, the function $ \phi_{F}(z)$ is also bounded away from $0$ on the disk $D_R$.  Thus we come to the conclusion that the functions $\phi_{Y}(z)$ and $\phi_{F}(z)$ are analytic on the disk $D_R$ and there exists a constant $c >0$ such that $\abs{ \phi_{Y}(z) },  \abs{ \phi_{F}(z) } \ge c >0$ for every $z \in D_R$. This implies that there exist two analytic functions $g_F$ and $g_Y$ such that
\[
\phi_{F}(z)=e^{g_F (z)}, \quad \phi_{Y}(z)=e^{g_Y (z)},\qquad\qquad z \in D_R,
\]
i.e., $g (z)= \log (\phi_{F}(z))$ and $g_Y(z)=\log(\phi_{Y}(z))$, for $z \in D_R$. In fact, the functions $g_F$ and $g_Y$ are given by the power series 
\begin{equation}\label{eq:lower-4}
g_F(z)=\sum_{\ell \ge 1}\frac{\kappa_\ell(F)}{\ell !}(iz)^\ell, \qquad\qquad g_Y(z)= \sum_{\ell \ge 1}\frac{\kappa_\ell (Y)}{\ell !}(iz)^\ell,
\end{equation}
respectively. Since the derivative of the analytic branch of the complex logarithm is  $(\log z)' = \frac{1}{z}$ (see \cite[Corollary $2.21$]{conway}), one can infer that for some constant $C>0$ whose value may differ from line to line and for every $z \in D_R$, we have 
\begin{align*}
\abs[\Big]{ \sum_{\ell \ge 2} \frac{\kappa_\ell (F) - \kappa_\ell (Y)}{\ell !}(iz)^\ell }  = \abs[\Big]{  \log (\phi_{F}(z)) - \log(\phi_{Y}(z)) } 
 \le C \abs[\Big]{ \phi_{F}(z) - \phi_{Y}(z) } \le C d_{\mathcal{H}_2} (F, Y).
\end{align*}
Now, using Cauchy's estimate for the coefficients of analytic functions, for any $\ell \ge 2$, we obtain
\[
\abs[\Big]{ \kappa_\ell (F) - \kappa_\ell (Y) } \le \ell! R^\ell \sup_{\abs{z} \le R} \abs[\Big]{ \log \phi_{F} (z) - \log \phi_{Y}(z) }.
\]
Therefore,
$$
\max \Big\{  \abs[\big]{ \kappa_\ell (F) - \kappa_\ell (Y) } \, : \, \ell=2,3,4,5,6 \Big\} \le {C} d_{\mathcal{H}_2} (F, Y)
$$
and the proof of Theorem \ref{thm:Main-thm} is complete.\hfill $\Box$

\subsection{Proof of Theorem \ref{thm:BT-OptimalRate}}\label{sec:proof:BaiTaqqu}

Suppose that the target random variable $Y$ is either as in item (a) or in item (b) of Theorem \ref{thm:BT-OptimalRate}, respectively. Then, for any $\ell \ge 3$, as $\gamma_1 \to -1/2$, it is known that
\begin{equation}\label{eq:BT-Cumulants-Difference}
\kappa_\ell (F_{\gamma_1,\gamma_2})  =  \kappa_\ell (Y) + O \left(  -\gamma_1 - \frac{1}{2}\right).
\end{equation}
In fact, when $Y$ is as in item (a), the asymptotic relation \eqref{eq:BT-Cumulants-Difference} for the cumulants has been established in \cite[Theorem 5.3]{Bai-Taqqu}, and in the case of item (b) in \cite[Lemma 3.2]{a-a-p-s}.   Hence, the result follows by applying our main Theorem \ref{thm:Main-thm}.\hfill $\Box$

\section*{Acknowledgement}
EA would like to thank Robert Gaunt for several stimulating discussions on Stein's method for Variance--Gamma approximation. The authors also thank him for his careful reading of our draft and his comment that simplified the original proof of Proposition \ref{prop:equivanishing-property}.


\begin{thebibliography}{AMMP16}
	
	\bibitem[AAPS17]{a-a-p-s}
	\textsc{Arras, B.,  Azmoodeh, E., Poly, G., and Swan, Y.} (2019).
	\newblock A bound on the 2-{W}asserstein distance between linear combinations
	of independent random variables.
	\newblock \emph{Stochastic processes and their Applications}, Vol. 129, Issue 7, 2341--2375.
	
	\bibitem[ACP14]{a-c-p}
	\textsc{Azmoodeh, E.,  Campese, S., and Poly, G.} (2014).
	\newblock Fourth {M}oment {T}heorems for {M}arkov diffusion generators.
	\newblock {\em J. Funct. Anal.}, 266(4):2341--2359.
	
	\bibitem[AEK20]{Optimal-Gamma}
	\textsc{Azmoodeh, E.,  Eichelsbacher, P., Knichel, L.} (2020).
	\newblock Optimal Gamma Approximation on  Wiener Space.
	\newblock {\em ALEA, Lat. Am. J. Probab. Math. Stat.} 17, 101--132.
	
	
\bibitem[AMPS17]{a-m-p-s}
\textsc{Arras, B., Mijoule, G., Poly, G., Swan, Y.} (2017).
\newblock A new approach to the {S}tein-{T}ikhomirov method: with applications to the second {W}iener chaos and {D}ickman convergence, 
\newblock 	arXiv:1605.06819v2.




\bibitem[AP17]{MS-survey}
\textsc{Azmoodeh, E., Peccati, G.} (2017).
\newblock Malliavin-Stein Method: a Survey of Recent Developments.
\newblock https://arxiv.org/abs/1809.01912.




	\bibitem[APP15]{a-p-p}
	\textsc{Azmoodeh, E., Peccati, G.,  Poly, G.} (2015).
	\newblock Convergence towards linear combinations of chi-squared random
	variables: a {M}alliavin-based approach.
	\newblock In {\em In memoriam {M}arc {Y}or---{S}\'eminaire de {P}robabilit\'es
		{XLVII}}, volume 2137 of {\em Lecture Notes in Math.}, pages 339--367.
	Springer, Cham.
	
	\bibitem[BT17]{Bai-Taqqu}
	\textsc{Bai, S.,  Taqqu, M.} (2017). 
	\newblock Behavior of the generalized Rosenblatt process at extreme critical exponent values, 
	\newblock \emph{Ann. Probab.} 45 (2), 1278-1324.
	
	
	
	
	\bibitem[BBNP12]{OptBerryEsseenRates}
	\textsc{Bierm{\'e}, H., Bonami, A., Nourdin, I., Peccati, G}. (2012).
	\newblock Optimal {B}erry-{E}sseen rates on the {W}iener space: the barrier of third and fourth cumulants.
	\newblock {\em ALEA Lat. Am. J. Probab. Math. Stat.}, 9(2):473--500.
	
	
	\bibitem[CGS11]{ChenGoldsteinShao2011}
	 \textsc{Chen, L. H. Y.,  Goldstein, L., and Q.-M. Shao,  Q.-M.} (2011).
	\newblock {\em Normal {A}pproximation by {S}tein's {M}ethod}.
	\newblock Probability and its Applications (New York). Springer, Heidelberg.
	
\bibitem[Con95]{conway}
\textsc{Conway, J. B.} (1995).
\newblock {\em Functions of one complex variable. {II}}, volume 159 of {\em
Graduate Texts in Mathematics}.
\newblock Springer-Verlag, New York.
	


\bibitem[DGV17]{Iterative-Stein}
\textsc{D\"obler, C.,   Gaunt, R. E., and   Vollmer, S. J.} (2017).  
\newblock An iterative technique for bounding derivatives of
solutions of Stein equations.
\newblock {\em Electron. J. Probab}. 22, no. 96, 1--39, 2017.
	
	

\bibitem[DP18]{d-p}
\textsc{D\"obler, C.,  Peccati, P.} (2018).
\newblock The {G}amma {S}tein equation and noncentral de {J}ong theorems.
\newblock {\em Bernoulli}, 24(4B):3384--3421.
	


\bibitem[ET15]{VarianceGammaPaper}
\textsc{Eichelsbacher, P., Th{\"a}le, C.} (2015).
\newblock Malliavin-{S}tein method for Variance-Gamma approximation on {W}iener space.
\newblock {\em Electron. J. Probab.}, 20:Paper No. 123, 28.
	

\bibitem[Gau13]{Robert-Thesis}
\textsc{Gaunt, R. E.} (2013).
\newblock Rates of Convergence of Variance-Gamma Approximations via Stein’s Method.
\newblock D.Phil. thesis, University of Oxford.
	
	
	
\bibitem[Gau14]{g-variance-gamma}
\textsc{Gaunt, R. E.} (2014). 
\newblock Variance-Gamma approximation via Stein's method.
\newblock \emph{Electron. J. Probab.}  19(38), pp.1--33.
	
\bibitem[Gau17]{gaunt-normal-product}
\textsc{Gaunt, R.E.} (2017) 
\newblock On Stein's method for products of normal random variables
and zero bias couplings. 
\newblock \emph{Bernoulli}, 23(4B), 3311-3345.
	
	
	
\bibitem[Gau18]{robert-asymptotics}
\textsc{Gaunt, R. E.} (2018).
\newblock  Inequalities for integrals of modified Bessel functions and expressions involving them.
\newblock \emph{J. Math. Anal. Appl}. 462, pp. 172--190.


\bibitem[Gau20a]{gaunt-vg-kol}
\textsc{Gaunt, R. E.} (2020).
\newblock  Wasserstein and Kolmogorov error bounds for Variance-Gamma approximation via Stein's method I. 
\newblock \emph{J. Theoret. Probab.} 33, pp. 465-505.
	
	
\bibitem[Gau20b]{gaunt-VG-perfect-bounds}
\textsc{R. E. Gaunt}. (2020) 
\newblock Stein factors for Variance-Gamma approximation in the Wasserstein and Kolmogorov distances. 
\newblock  \texttt{https://arxiv.org/pdf/2008.06088.pdf}.
	
	
	
	
\bibitem[KKP01]{Laplace-bible}
\textsc{Kotz, S., Kozubowski, T. J. and Podg\'orski, K.} (2001).
\newblock \emph{The Laplace Distribution and Generalizations: A Revisit with New Applications}. Springer.

	
\bibitem[Luk70]{lukas}
\textsc{Lukacs, E.} (1970).
\newblock {\em Characteristic functions}.
\newblock Hafner Publishing Co., New York.
\newblock Second edition, revised and enlarged.
	


\bibitem[MCC98]{Madan98}
\textsc{Madan, D. B., Carr, P. and Chang, E. C.}  (1998).  
\newblock  The Variance Gamma process and option pricing. 
\newblock \emph{Eur. Finance Rev.} 2, pp. 74-105.

\bibitem[MS90]{Madan90}
\textsc{Madan, D. B., Seneta, E.} (1990).
\newblock  The Variance Gamma (V.G.) Model for Share Market Returns. 
\newblock \emph{J. Bus.} 63, pp. 511-524.
	
	
	
	
	
\bibitem[Meg98]{Banach-Space-Theory}
\textsc{Megginson, R.~E.} (1998).
\newblock {\em An introduction to Banach space theory}, volume 183 of {\em
Graduate Texts in Mathematics}.
\newblock Springer-Verlag, New York.
	
\bibitem[NN18]{Nua-Nua}
\textsc{Nualart, D.,  Nualart, E.} (2018).
\newblock {\em Introduction to {M}alliavin calculus}, volume~9 of {\em
Institute of Mathematical Statistics Textbooks}.
\newblock Cambridge University Press, Cambridge.
	


\bibitem[NP05]{FmtOriginalReference}
\textsc{Nualart, D., Peccati, G.} (2005).
\newblock Central limit theorems for sequences of multiple stochastic
integrals.
\newblock {\em Ann. Probab.}, 33(1):177--193.


	
\bibitem[NP09a]{n-p-noncentral}
\textsc{Nourdin, I., Peccati, G.} (2009).
\newblock Noncentral convergence of multiple integrals.
\newblock {\em Ann. Probab.}, 37(4):1412--1426.
	
	

\bibitem[NP09b]{StMethOnWienChaos}
\textsc{Nourdin, I., Peccati, G.} (2009).
\newblock Stein's method on {W}iener chaos.
\newblock {\em Probab. Theory Related Fields}, 145(1-2):75--118.
	
	
	
\bibitem[NP10]{CumOnTheWienerSpace}
\textsc{Nourdin, I., Peccati, G.} (2010).
\newblock Cumulants on the {W}iener space.
\newblock {\em J. Funct. Anal.}, 258(11):3775--3791.
	
	
	
\bibitem[NP12a]{n-pe-1}
\textsc{Nourdin, I., Peccati, G.} (2012).
\newblock {\em Normal Approximations with Malliavin Calculus: From Stein's	Method to Universality}.
\newblock Cambridge Tracts in Mathematics. Cambridge University Press.
	
	
\bibitem[NP12b]{n-poly-2wiener}
\textsc{Nourdin, I.,  Poly, G.} (2012).
\newblock Convergence in law in the second {W}iener/{W}igner chaos.
\newblock {\em Electron. Commun. Probab.}, 17:no. 36, 12.
	
	
\bibitem[NP15]{n-p-optimal}
\textsc{Nourdin, I., Peccati, G.} (2015).
\newblock The optimal fourth moment theorem.
\newblock {\em Proc. Amer. Math. Soc.}, 143(7):3123--3133.
	
	
	
\bibitem[NPR10]{InvPrinForHomSums}
\textsc{Nourdin, I., Peccati, G., Reinert, G}. (2010).
\newblock Invariance principles for homogeneous sums: universality of
{G}aussian {W}iener chaos.
\newblock {\em Ann. Probab.}, 38(5):1947--1985.
	
	
	
\bibitem[NR14]{n-r}
\textsc{Nourdin, I.,  Rosi\'nski, J.} (2014).
\newblock Asymptotic independence of multiple {W}iener-{I}t\^o integrals and
the resulting limit laws.
\newblock {\em Ann. Probab.}, 42(2):497--526.
	
	
\bibitem[Nua06]{GelbesBuch}
\textsc{Nualart, D}. (2006).
	\newblock {\em The Malliavin calculus and related topics}.
	\newblock Probability and its applications. Springer, Berlin and Heidelberg and	New York, 2. ed. edition.
	

\bibitem[Rud91]{Rudin-Functional-Analysis}
\textsc{Rudin, W.} (1991).
\newblock {\em Functional Analysis}, 2nd Edition, New York: McGraw-Hill.
	
	
	
\bibitem[Ste72]{stein}
\textsc{Stein, C.} (1972).
\newblock A bound for the error in the normal approximation to the distribution of a sum of dependent random variables.
\newblock In {\em Proceedings of the {S}ixth {B}erkeley {S}ymposium on
{M}athematical {S}tatistics and {P}robability ({U}niv. {C}alifornia,
{B}erkeley, {C}alif., 1970/1971), {V}ol. {II}: {P}robability theory}, pages
	583--602. Univ. California Press, Berkeley, Calif.
	
	
\bibitem[Taq75]{Taqqu1975}
\textsc{Taqqu, M. S.} (1975).
\newblock  Weak convergence to fractional Brownian motion and to the Rosenblatt process.
\newblock  \emph{Z. Wahrsch. Verw. Gebiete} 31 287-302. 



\bibitem[ZP02]{HandBook-Exact-Solutions-ODE}
\textsc{Valentin F. Zaitsev and Andrei D. Polyanin} (2002).
\newblock  \emph{Handbook of Exact Solutions for Ordinary Differential Equations},
\newblock Chapman and Hall/CRC.	
	
\bibitem[Vau07]{Vaughn-Special-Functions}
\textsc{Vaughn, 	M. T.} (2007).
\newblock \emph{Introduction to Mathematical Physics}.  
\newblock Wiley-VCH.
	 
		
		
		
	
	
\end{thebibliography}
\end{document}